\documentclass[a4paper, 11pt]{article}
\pdfoutput=1

\title{The probability of generating finite and profinite groups}
\author{Scott Harper \& Martyn Quick}
\date{}

\usepackage[margin=2.5cm]{geometry}
\usepackage[UKenglish]{isodate}
\usepackage[T1]{fontenc}
\usepackage{microtype}
\cleanlookdateon
\usepackage{amsmath, amssymb, array, url, mathtools, stmaryrd, multicol}
\usepackage{newpxtext}
\usepackage[varg]{newpxmath}
\usepackage[scr=rsfso]{mathalpha}
\SetSymbolFont{stmry}{bold}{U}{stmry}{m}{n}
\usepackage[shortlabels]{enumitem}
\setlist[enumerate]{label={\normalfont(\roman*)}}
\usepackage[hidelinks]{hyperref}
\numberwithin{equation}{section}

\usepackage[amsmath, thmmarks]{ntheorem}
\theorembodyfont{\slshape}
\newtheorem{theorem}{Theorem}[section]
\newtheorem{corollary}[theorem]{Corollary}
\newtheorem{lemma}[theorem]{Lemma}
\newtheorem{proposition}[theorem]{Proposition}
\newtheorem{thm}{Theorem}
\newtheorem{cor}[thm]{Corollary}
\theorembodyfont{\normalfont}
\newtheorem{remark}[theorem]{Remark}
\newtheorem{example}[theorem]{Example}
\newtheorem{rem}[thm]{Remark}
\newtheorem{ex}[thm]{Example}

\makeatletter
\theoremstyle{nonumberplain}
\theorembodyfont{\normalfont}
\theoremheaderfont{\normalfont\scshape}
\theoremsymbol{~\ensuremath\square}
\newtheoremstyle{proofstyle}%
  {\item[\theorem@headerfont\hskip\labelsep ##1\theorem@separator]}%
  {\item[\theorem@headerfont\hskip\labelsep ##3\theorem@separator]}
\theoremstyle{proofstyle}
\newtheorem{proof}{Proof:}
\makeatother

\newcommand{\AND}{\qquad\text{and}\qquad}

\newcommand{\im}{\operatorname{im}}
\newcommand{\<}{\langle}
\renewcommand{\>}{\rangle}
\renewcommand{\leq}{\leqslant}
\renewcommand{\geq}{\geqslant}

\renewcommand{\nleq}{\nleqslant}
\newcommand{\normal}{\trianglelefteqslant}

\newcommand{\padics}[1][p]{\mathbb{Z}_{#1}}
\newcommand{\semidirect}{\mathbin{\vcentcolon}}
\newcommand{\nbd}{\nobreakdash-}
\newcommand{\set}[2]{\{\,#1\mid#2\,\}}

\DeclarePairedDelimiter{\order}{\lvert}{\rvert}
\let\wreath\wr
\renewcommand{\wr}[1][]{\mathbin{\wreath_{#1}}}

\newcommand{\Aut}[1]{\operatorname{Aut}#1}

\newcommand{\Cent}[2]{C_{#1}(#2)}
\newcommand{\Centre}[1]{Z(#1)}
\newcommand{\Cohom}[1]{H^{1}(#1)}
\newcommand{\Frat}[1]{\Phi(#1)}
\newcommand{\half}{\tfrac{1}{2}}
\newcommand{\Inn}[1]{\operatorname{Inn}#1}
\newcommand{\Norm}[2]{N_{#1}(#2)}
\newcommand{\Op}[2][p]{O_{#1}(#2)}
\newcommand{\Out}[1]{\operatorname{Out}#1}
\newcommand{\Sym}[1]{\operatorname{Sym}(#1)}
\newcommand{\wAut}[1]{\operatorname{\widetilde{Aut}}#1}

\renewcommand{\epsilon}{\varepsilon}
\newcommand{\C}{\mathscr{C}}
\newcommand{\D}{\mathscr{D}}
\newcommand{\K}{\mathcal{K}}
\newcommand{\M}{\mathscr{M}}
\newcommand{\T}{\mathscr{T}}

\newcommand{\GL}{\mathrm{GL}}
\newcommand{\SL}{\mathrm{SL}}
\newcommand{\PGL}{\mathrm{PGL}}
\newcommand{\PSL}{\mathrm{PSL}}
\newcommand{\PSU}{\mathrm{PSU}}
\newcommand{\PSp}{\mathrm{PSp}}
\newcommand{\POm}{\mathrm{P}\Omega}
\newcommand{\F}{\mathbb{F}}
\newcommand{\FF}{\overline{\mathbb{F}}}

\begin{document}

\maketitle
\def\abstractname{\vspace{-2\baselineskip}}

\begin{abstract}
\noindent
Famously, every finite simple group~$G$ can be generated by a pair of elements.
Moreover, Liebeck and Shalev (1995) proved that the probability that a pair of elements generate~$G$ tends to $1$ as $|G| \to \infty$.
In this paper, we generalize this theorem of Liebeck and Shalev.
Work of Lucchini and Menegazzo (1997) implies that a finite group~$G$ can be generated by a pair of elements if it has a unique chief series.
As a consequence of our main theorem, the probability that a pair of elements generate such a group~$G$ tends to~$1$ as $|S| \to \infty$, where $S$ is the unique simple quotient of~$G$.
We also prove that a profinite group~$G$ with finitely many chief series has a topological generating set of size~$d < \infty$, and for any such~$d$, the probability that a $d$\nbd tuple of elements topologically generates~$G$ is positive; moreover, we can take $d = 2$ if $G$ has a unique chief series.
Along the way, we show that the chief factors of a finite group with a unique chief series are highly constrained, and we also analyze the maximal subgroup zeta function of a finite group with a unique minimal normal subgroup.
\end{abstract}

\section{Introduction}
\label{s:intro}

Group generation is a subject with a long history, beginning with some very elementary questions.
For instance, the finite symmetric and alternating groups are well-known to be generated by just two elements, but, as early 1882, Netto \cite{Netto82} wrote 
\begin{quote}
If we arbitrarily select two or more [permutations] of $n$ elements, it is to be regarded as extremely probable that the group of lowest order which contains these is the symmetric group, or at least the alternating group.
\end{quote}
Netto's Conjecture was proved in a precise sense by Dixon in 1969 \cite{Dixon69}, who replaced it with an even bolder conjecture.
As a consequence of the Classification of Finite Simple Groups, we know that every finite simple group can be generated by a pair of elements \cite{Steinberg62, AschbacherGuralnick84}, but even before the Classification was complete, Dixon conjectured that almost all pairs of elements of a finite simple group generate the whole group, generalizing the phenomenon seen for alternating groups.
Liebeck and Shalev \cite{LiebeckShalev95} (building on earlier work of Kantor and Lubotzky \cite{KantorLubotzky90}) proved Dixon's Conjecture.
To state their result precisely, for a finite group~$G$ and a positive integer~$d$, let $P_d(G)$ be the probability that a uniform random $d$\nbd tuple $(g_1, \dots, g_d) \in G^d$ satisfies $\< g_1, \dots, g_d \> = G$.
With this notation, Liebeck and Shalev's theorem states that if $G$~is a finite simple group, then $P_2(G) \to 1$ as $\order{G} \to \infty$.
(For more background on this subject, see Burness' survey \cite{Burness19}.)

The aim of this paper is to generalize this landmark result on probabilistic generation to a much broader class of groups.
We call a finite group \emph{uniserial} if it has a unique chief series. 
Clearly, all finite simple groups are uniserial.
As a consequence of a result of Lucchini and Menegazzo \cite[Theorem~1.1]{LucchiniMenegazzo97}, every uniserial group can be generated by a pair of elements (see Corollary~\ref{cor:lucchini_menegazzo}). 
In this paper, we establish the corresponding probabilistic result for these groups.

\begin{thm}
\label{thm:uniserial}
For all $\epsilon > 0$, there exists $c > 0$ such that if $G$~is a finite uniserial group and $N$~is a normal subgroup with $\order{G/N} \geq c$, then
\[
P_2(G) \geq (1 - \epsilon) \, P_2(G/N).
\]
\end{thm}

\begin{cor}
\label{cor:uniserial}
Let $G$ be a finite nontrivial uniserial group, and let $S$ be the unique simple quotient of~$G$.
Then $P_2(G) \to 1$ as $\order{S} \to \infty$.
\end{cor}

\clearpage

\begin{rem}
\label{rem:uniserial}
The generality of Theorem~\ref{thm:uniserial} means we can consider sequences of uniserial groups where the order of unique simple quotient does not tend to infinity.
For example, Dixon \cite{Dixon69} proved that $P_2(S_n) \to 3/4$ as $n \to \infty$, so Theorem~\ref{thm:uniserial} implies that if $G$ is a uniserial group with $S_n$ as a quotient, then $P_2(G) \to 3/4$ as $n \to \infty$ (see Corollary~\ref{cor:limit} for a general statement).
\end{rem}

The class of uniserial groups is very broad. 
For example, it contains all the groups that can be obtained by iteratively applying the two constructions in the following example.

\begin{ex}
\label{ex:main}
Let $G$ be a uniserial finite group.
\begin{enumerate}
\item 
Let $p$ be prime, and let $V$ be a finite-dimensional $\F_pG$-module.
If $V$ is a uniserial module (i.e., if $V$ has a unique composition series as an $\F_pG$-module), then the affine group $V\semidirect G$ is uniserial provided that $G$~acts faithfully on the quotient of~$V$ by its maximal submodule (see Lemma~\ref{lem:affine_construction}).
In particular, $V\semidirect G$ is uniserial if $V$ is a faithful irreducible $\F_{p}G$\nbd module.

\item 
Let $T$ be a finite group, and let $G$ act faithfully and transitively on a finite set~$\Omega$.
If $T$ is a nonabelian simple group, then the wreath product $T \wr_\Omega G$ is uniserial (see Lemma~\ref{lem:wreath_construction}).
In particular, iterated wreath products of finite nonabelian simple groups in faithful transitive actions are uniserial, so Theorem~\ref{thm:uniserial} generalizes \cite[Theorem~A]{Quick06} (which generalizes the main results of \cite{Bhattacharjee94} and \cite{Quick04}).
In general, $T \wr_\Omega G$ may or may not be uniserial.
For example, with respect to the natural action of $S_n$ on $n \geq 5$ points, $C_p \wr S_n$ is not uniserial, but if $d$ is a prime divisor of $q - 1$, then $\SL_d(q) \wr S_n$ is uniserial (see Examples~\ref{ex:affine_symmetric} and~\ref{ex:wreath_quasisimple}).
\end{enumerate}
\end{ex}

Theorem~\ref{thm:uniserial} is a consequence of a more general result, as we now explain.
Let $G$ be a finite group.
We say that $G$ has a \emph{prescribed chief tail} $\T = (N_0, \dots, N_\ell)$ if every chief series $G = G_0 > \cdots > G_n = 1$ satisfies $G_{n - i} = N_{\ell - i}$ for all $i \in \{0, \dots, \ell\}$, or, said otherwise, every chief series of~$G$ ends with $\T$.
Observe that $(1)$ is a prescribed chief tail of any finite group, $(N, 1)$ is a prescribed chief tail of $G$ if and only if $N$ is the unique minimal normal subgroup of~$G$ and $G$ has a prescribed chief tail $(N_0, \dots, N_\ell)$ satisfying $N_0 = G$ if and only if $G$ is uniserial.
We can now state the main theorem.

\begin{thm}
\label{thm:main}
For all $\epsilon > 0$, there exists $c > 0$ such that if $G$ is a finite group with a prescribed chief tail~$\T$, $N \in \T$ satisfies $\order{G/N} \geq c$ and $d \geq 2$, then
\[
P_d(G) \geq (1 - \epsilon) \, P_d(G/N).
\]
\end{thm}

\begin{rem}
\label{rem:main}
In Theorem~\ref{thm:main}, the hypothesis that $d \geq 2$ is necessary
since one can easily construct groups~$G$ with a unique minimal normal
subgroup~$N$ with $\order{G/N}$ arbitrarily large while still having
$P_{1}(G) = 0$ and $P_{1}(G/N) > 0$.
One such example is $G = A_5 \wr C_n$ and $N = A_5^n$ for all $n \geq 1$.
\end{rem}

We now turn to profinite groups and begin by extending some notation from finite groups.
For a profinite group~$G$, let $P_d(G)$ be the probability (with respect to the normalized Haar measure on~$G$) that a random $d$\nbd tuple $(g_1, \dots, g_d) \in G^d$ topologically generates~$G$ (i.e., $\< g_1, \dots, g_d \>$ is dense in~$G$), and let $d(G)$ be the smallest~$d$ such that $G$ has a topological generating $d$\nbd tuple.
Finally, a \emph{chief series} of a profinite group~$G$ is an unrefinable series $G = G_0 > G_1 > \cdots$ of open normal subgroups of~$G$.
We can now give our application of Theorem~\ref{thm:main} to profinite groups.

\begin{thm}
\label{thm:profinite}
Let $G$~be a profinite group with finitely many chief series.
Then $d(G)$ is finite and $P_d(G) > 0$ for all $d \geq d(G)$.
Moreover, if $G$ has a unique chief series, then $d(G) \leq 2$, so $P_2(G) > 0$.
\end{thm}

Theorem~\ref{thm:profinite} implies the weaker statement that a profinite group with finitely many chief series is positively finitely generated (PFG), which could alternatively be deduced from the characterization of PFG groups in \cite[Theorem~4]{JaikinZapirainPyber11}.

An important ingredient in our proof is a generalization of the maximal subgroup zeta function introduced by Liebeck and Shalev in \cite{LiebeckShalev96}.
Let $G$ be a finite group.
For a normal subgroup~$N$, write $\M(G, N)$ for the set of maximal subgroups of~$G$ that do not contain~$N$.
For each real number~$s > 1$, write
\[
\zeta_{G, N}(s) = \sum_{M \in \M(G, N)} \order{G : M}^{-s} = \sum_n \frac{m_n(G, N)}{n^s},
\]
where $m_n(G, N)$ is the number of maximal subgroups $M \in \M(G, N)$ of index~$n$.
If $N = G$, then $\M(G, N)$ is just the set of maximal subgroups of $G$, so $\zeta_{G, G}(s)$ is just the zeta function $\zeta_{G}(s)$ introduced by Liebeck and Shalev in \cite{LiebeckShalev96} and studied by Liebeck, Martin and Shalev \cite{LiebeckMartinShalev05}, who proved that if $s > 1$, then  $\zeta_G(s) \to 0$ as $\order{G} \to \infty$ when $G$ is simple.

By work of Gasch\"utz \cite{Gaschutz59} (see Lemma~\ref{lem:prob}), for all normal subgroups~$N$ of~$G$ and positive integers $d \geq d(G)$,
\[
\frac{P_d(G)}{P_d(G/N)} \geq 1 - \zeta_{G, N}(d),
\]
where $P_d(G)/P_d(G/N)$ can be interpreted as the conditional probability $P_d(G, N)$ that a $d$\nbd tuple of elements of~$G$ generate $G$ given that their images in $G/N$ generate $G/N$.

We now outline the proof of Theorem~\ref{thm:main}.
Let $G$ be a finite group with a prescribed chief tail $(N_0, \dots, N_\ell)$ and let $N = N_t \in \T$.
For $d \geq 2$ such that $G$ is $d$\nbd generated,
\begin{equation}
\label{eq:main}
\frac{P_d(G)}{P_d(G/N)} = \prod_{i = t+1}^{\ell} \frac{P_d(G/N_{i})}{P_d(G/N_{i-1})} \geq \prod_{i = t+1}^{\ell} \bigl(1 - \zeta_{G/N_i, N_{i-1}/N_i}(d) \bigr).
\tag{$\star$}
\end{equation}
This leads to the two main main aspects of our proof.

\paragraph{Maximal subgroup zeta function}

Let $G$ be a finite group with a unique minimal normal subgroup $N$ and consider the zeta function~$\zeta_{G,N}(s)$.
We do this since $N_{i-1}/N_{i}$ is the unique minimal normal subgroup of~$G/N_{i}$ in each term of \eqref{eq:main}.

When $s = 1$, the value $\zeta_{G, N}(1)$ is nothing other than the number of conjugacy classes of maximal subgroups of~$G$ that do not contain~$N$ (see Lemma~\ref{lem:zeta_classes}).
To describe $\zeta_{G, N}(s)$ when $s > 1$, we refer to a real-valued function~$\alpha$ defined on the set of finite simple groups. 
This function is defined in Section~\ref{ss:p_function}, and for now we just note that there is an absolute constant~$c > 0$ such that $\alpha(T) \geq c$ for all~$T$ and $\alpha(T) \to \infty$ as $\order{T} \to \infty$.
When $T$ is a cyclic group of prime order, $\alpha(T) = \order{T}$.

\begin{thm}
\label{thm:zeta}
Let $G$ be a finite group with a unique minimal normal subgroup~$N$.
Write $N = T^n$ where $T$ is simple.
Let $d \geq 2$ be an integer.
Then 
\[
\zeta_{G, N}(d) < \alpha(T)^{-n(d-\iota)},
\]
where $\iota = \frac{3}{2}$ if $T$ is abelian and $\iota = 1$ otherwise.
In particular, $\zeta_{G, N}(d) \to 0$ as $\order{N} \to \infty$.
\end{thm}

Theorem~\ref{thm:zeta} has the following immediate consequence.

\begin{cor}
\label{cor:zeta}
Let $d \geq 2$. 
Let $G$ be a finite $d$\nbd generated group with a unique minimal normal subgroup~$N$.
Write $N = T^n$ where $T$ is simple.
Then 
\[
P_d(G, N) > 1 - \alpha(T)^{-n(d-\iota)},
\] 
where $\iota = \frac{3}{2}$ if $T$ is abelian and $\iota = 1$ otherwise.
\end{cor}

Corollary~\ref{cor:zeta}, in particular, implies that $P_d(G, N) \to 1$ as $\order{N} \to \infty$, a fact first proved by Lucchini and Morini \cite{LucchiniMorini02}.
However, to prove Theorem~\ref{thm:main}, we need the stronger asymptotic estimate in Theorem~\ref{thm:zeta} and not just the fact that $P_d(G, N) \to 1$.
That said, combining the observation that $P_d(G, N) \to 1$ as $\order{N} \to \infty$ with Theorem~\ref{thm:main}, which gives $P_d(G, N) \to 1$ as $\order{G/N} \to \infty$, yields the following result that generalises the theorem of Lucchini and Morini.

\begin{cor}
\label{cor:monolthic}
Let $d \geq 2$. 
Let $G$ be a finite $d$\nbd generated group with a unique minimal normal subgroup~$N$. 
Then $P_d(G, N) \to 1$ as $\order{G} \to \infty$.
\end{cor}

Detomi and Lucchini \cite{DetomiLucchini13} also proved that $P_{d}(G,N) \geq 53/90$, and, as explained in the proof of Proposition~\ref{prop:zeta_nonabelian}, two cases in our proof of Theorem~\ref{thm:zeta} follow the strategy in \cite{DetomiLucchini13}, but the bulk of our proof is dedicated to a third case, which requires an intricate analysis of product actions and detailed information about almost simple groups.

\paragraph{Widths of chief factors in prescribed chief tails}
Returning to \eqref{eq:main}, for each $i \in \{1, \dots, \ell \}$, write $N_{i-1}/N_i = T_i^{n_i}$ where $T_i$ is simple.
We call $n_i$ the \emph{width} of the chief factor $T_i^{n_i}$.
In view of the bound in Theorem~\ref{thm:zeta}, the second main aspect of our proof concerns the widths of the chief factors in a prescribed chief tail.
In the previous special cases that have been considered, such as uniserial groups constructed as iterated wreath products of alternating groups in \cite{Bhattacharjee94}, the sequence $(n_1, n_2, \dots)$ is strictly increasing.
However, in our general setting, this is not the case, as the following example highlights.
In this example, we call a chief factor $N_{i-1}/N_i$ \emph{Frattini} if $N_{i-1}/N_i \leq \Frat{G/N_i}$ and \emph{non-Frattini} otherwise.

\begin{ex} Let $p$ be an odd prime.
\label{ex:width}
\begin{enumerate}
\item 
\label{it:width_almost_simple}
Let $G = \Aut(\POm_8^+(p))$.
Since $\Out(\POm_8^+(p)) \cong S_4$ (see \cite[p.~38]{KleidmanLiebeck90} for example), $G$ is uniserial with chief factors 
\[
C_2 \quad C_3 \quad C_2^2 \quad \POm_8^+(p),
\]
all of which are non-Frattini.
\item
\label{it:width_affine}
Example~\ref{ex:affine_equality} considers the affine group $G = p^4 \semidirect H$ where $p \equiv \pm 1 \pmod{8}$ and $H \leq \GL_4(p)$ is given as the full preimage of $A_4^2 \semidirect C_4 \leq S_4 \wr S_2 \leq \PGL_4(p)$, and we prove that $G$ is uniserial with chief factors
\[
C_2 \quad (C_2) \quad C_3^2 \quad C_2^4 \quad (C_2) \quad C_p^4,
\]
where the Frattini chief factors are in brackets. 
\end{enumerate}
\end{ex}

By Example~\ref{ex:width}\ref{it:width_affine}, we need to exclude Frattini chief factors to ensure that the widths do not strictly decrease, and by part~\ref{it:width_almost_simple} the widths can still strictly decrease even when we do (by part~\ref{it:width_affine} they can also stay constant).
Despite this, the result following shows that the widths of the non-Frattini chief factors in a prescribed chief tail do usually strictly increase, with Example~\ref{ex:width} capturing the main exceptions.

\begin{thm}
\label{thm:width_simple}
Let $G$~be a finite group with a prescribed chief tail $\T = (N_0, \dots, N_\ell)$.
Let $N_{i-1}/N_i = S^m$ and $N_{j-1}/N_j = T^n$ be consecutive non-Frattini chief factors in $\T$, with $S$ and $T$  simple and $i < j$.
Then
\begin{enumerate}
\item $n \geq m$ unless $S = C_2$ and $T = \mathrm{P}\Omega_8^+(q)$
  for an odd prime power $q$
\item $n > m$ unless $S$ is cyclic and $T$ is nonabelian simple, or
  $S$ and $T$ are cyclic and $m = n \leq 4$.
\end{enumerate}
\end{thm}

Theorem~\ref{thm:width_simple} is a simplified version of Theorem~\ref{thm:width}, which gives more precise information on the chief factors in prescribed chief tails.
The proof of this result involves, among other things, representation-theoretic arguments to control the structure of linear groups, inspired by \cite{FeitTits78}.

\begin{paragraph}{Notation}
Our notation for finite simple groups follows \cite{KleidmanLiebeck90}.
The \emph{rank} of a group of Lie type refers to its untwisted rank, i.e., the rank of the ambient algebraic group.
For a finite group $G$, we write $\Frat{G}$ for the Frattini subgroup of $G$, and for a prime $p$, we write $\Op{G}$ for the largest normal $p$-subgroup of $G$.
All logarithms are base-two.
\end{paragraph}

\section{Preliminaries}
\label{s:prelims}

\subsection{Generation}
\label{ss:p_prob}

The following well-known lemma is a simplified version of a result originally due to Gasch\"utz \cite[Satz~1]{Gaschutz59}. 
Given the centrality of this lemma to this paper, we provide the proof.
Recall that for a group $G$ and a normal subgroup $N \normal G$, we write $\M(G,N)$ for the set of maximal subgroups of $G$ that do not contain $N$.

\begin{lemma}
\label{lem:prob}
Let $G$ be a finite $d$\nbd generated group, and let $N$ be a normal subgroup of $G$.
Then
\[
P_d(G) \geq \biggl( 1 - \sum_{M \in \M(G, N)} \order{G : M}^{-d} \biggr) \, P_d(G/N).
\]
\end{lemma}

\begin{proof}
Write 
\begin{align*}
X &= \set{ (g_1, \dots, g_d) \in G^d}{\< g_1, \dots, g_d \> = G } \\
Y &= \set{ (Ng_1, \dots, Ng_d) \in (G/N)^d}{\< Ng_1, \dots, Ng_d \> = G/N }.
\end{align*}
The map $\phi \colon X \to Y$ defined as $(g_1, \dots, g_d) \mapsto
(Ng_1, \dots, Ng_d)$ is surjective and Gasch\"utz's Lemma \cite[Satz~1]{Gaschutz55} implies that the size of the preimage $(Ng_1, \dots, Ng_d)\phi^{-1}$ is independent of the choice of $(Ng_1, \dots, Ng_d) \in G/N$.
Fix $g = (g_1, \dots, g_d) \in G^d$ such that $\< g_1, \dots, g_d \> = G$.
Then
\[
P_d(G) = \frac{\order{X}}{\order{G}^d} = \frac{\order{(Ng_1, \dots,
    Ng_d)\phi^{-1}} \order{Y}}{\order{G}^d} = \frac{\order{(Ng_1, \dots, Ng_d)\phi^{-1}}}{\order{N}^d} \cdot P_d(G/N).
\]

Let $n = (n_1, \dots, n_d) \in N^d$. If $ng \notin (Ng_1, \dots, Ng_d)\phi^{-1}$, then $\< n_1g_1, \dots, n_dg_d \> \neq G$, so there exists a maximal subgroup $M$ of $G$ such that $ng \in M^d$.  Since $\< Nn_1g_1, \dots, Nn_dg_d \> = G/N$, necessarily $N \nleq M$, so $n \in N^d \cap M^dg^{-1}$ for some $M \in \M(G, N)$.
Therefore, 
\[
1 - \frac{\order{(Ng_1, \dots, Ng_d)\phi^{-1}}}{\order{N}^d} = \frac{\order{N}^d - \order{(Ng_1, \dots, Ng_d)\phi^{-1}}}{\order{N}^d} \leq \sum_{M \in \M(G, N)} \frac{\order{N^d \cap M^dg^{-1}}}{\order{N}^d}.
\]

For all $M \in \M(G, N)$, since $N \nleq M$, we know that $G = NM$, so we can fix $x \in N^d$ such that $xg \in M^d$, but then $(N^d \cap M^dg^{-1})x^{-1} = N^dx^{-1} \cap M^d(xg)^{-1} = (N \cap M)^d$.
Therefore,
\[
\frac{\order{N^d \cap M^d g^{-1}}}{\order{N}^d} = \frac{\order{N \cap M}^d}{\order{N}^d} =  \order{N: N \cap M}^{-d} = \order{G : M}^{-d}.
\]

Drawing these conclusions together establishes the bound in the statement.
\end{proof}

One consequence of Lemma~\ref{lem:prob} is the following well-known result.

\begin{corollary}
\label{cor:prob_frattini}
Let $d$ be a positive integer.
Let $G$ be a finite group and let $N$ be a normal subgroup of $G$.
Then $P_d(G) \leq P_d(G/N)$ with equality if $N \leq \Frat{G}$.
\end{corollary}

\begin{proof}
If $\< g_1, \dots, g_d \> = G$, then $\< Ng_1, \dots, Ng_d \> = G/N$, so $P_d(G/N) \geq P_d(G)$.
If $N \leq \Frat{G}$, then every maximal subgroup of $G$ contains $N$, so $P_d(G) \geq P_d(G/N)$, by  Lemma~\ref{lem:prob}.
\end{proof}

The following result, which is the main theorem of \cite{LucchiniMenegazzo97}, is a key tool for us.

\begin{theorem}[Lucchini \& Menegazzo, 1997]
\label{thm:lucchini_menegazzo}
Let $G$ be a noncyclic finite group with a unique minimal normal subgroup $N$.
Then $d(G) = \max\{ 2, d(G/N) \}$.
\end{theorem}

Theorem~\ref{thm:lucchini_menegazzo} naturally yields a generalization to groups with a prescribed chief tail.

\begin{corollary}
\label{cor:lucchini_menegazzo}
Let $G$ be a noncyclic finite group with a prescribed chief tail $(N_0, \dots, N_\ell)$.
Then $d(G) = \max\{ 2, d(G/N_0) \}$.
In particular, every finite uniserial group is $2$-generated.
\end{corollary}

\begin{proof}
For each $i \in \{ 1, \dots, \ell \}$, the subgroup $N_{i-1}/N_i$ is the unique minimal normal subgroup of $G/N_i$, and $(G/N_i)/(N_{i-1}/N_i) \cong G/N_{i-1}$.
Therefore, repeatedly applying Theorem~\ref{thm:lucchini_menegazzo} gives
\[
d(G) = \max\{ 2, d(G/N_{\ell-1}) \} = \max\{ 2, d(G/N_{\ell-2}) \} = \cdots = \max\{ 2, d(G/N_0) \},
\]
as required. 
If $G$ is uniserial, then $G$ has a prescribed chief tail $(N_0, \dots, N_{\ell})$ where $G = N_0$, so either $G$ is cyclic and $d(G) = 1$, or, by the previous observation, $d(G) = \max\{ 2, d(G/G) \} = 2$.
\end{proof}

\subsection{Abelian minimal normal subgroups}
\label{ss:p_minimal_normal}

\begin{lemma}
\label{lem:complemented}
Let $G$ be a finite group with an abelian minimal normal subgroup $N$.
Then $N \nleq \Frat{G}$ if and only if $N$ has a complement in $G$.
\end{lemma}

\begin{proof}
First assume that $N \nleq \Frat{G}$. 
Then $G$ has a maximal subgroup $M$ not containing $N$.
Then $G = MN$ and $M \cap N$ is normal in $G = MN$ since it is normalized by $M$ and $N$ is abelian. 
As $N \nleq M$ and $N$ is a minimal normal subgroup, we deduce that $M \cap N = 1$, so $M$ is a complement to $N$ in $G$.

Conversely, assume that $H$ is a complement to $N$ in $G$.
Then $H$ is contained in a maximal subgroup $M$ of $G$, and hence $N \nleq M$ since otherwise $G = NH \leq M$, so $N \nleq \Frat{G}$, as claimed.
\end{proof}

\begin{lemma}
\label{lem:faithful}
Let $G$ be a finite group with a unique minimal normal subgroup $N$ and suppose $N$~is abelian.
Assume that $\Frat{G} = 1$.
Then $\Cent{G/N}{N} = 1$.
\end{lemma}

\begin{proof}
Since $\Frat{G} = 1$, by Lemma~\ref{lem:complemented} there is a complement~$H$ for~$N$.
Consequently, $H \cong G/N$ and $\Cent{H}{N} \cong \Cent{G/N}{N}$.
Now $\Cent{H}{N}$ is normal in $H$ and centralized by $N$, so $\Cent{H}{N}$ is a normal subgroup of $HN = G$.
However, $\Cent{H}{N} \cap N \leq H \cap N = 1$ and $N$ is the unique minimal normal subgroup of $G$, so $\Cent{G/N}{N} \cong \Cent{H}{N} = 1$, as required.
\end{proof}

\subsection{The function \texorpdfstring{\boldmath $\alpha$}{α}}
\label{ss:p_function}

We now define the function $\alpha$ that features in Theorem~\ref{thm:zeta}.
For a finite simple group $T$, write
\begin{equation}
\label{eq:f}
\alpha(T) = \left\{ 
\begin{array}{ll}
\alpha^*(T)                                                                                        & \text{if $T$ is abelian} \\
\left(\alpha^*(T)^{-2} + \tfrac{61}{60}\left(\half\order{T}^{3/8}\right)^{-2}\right)^{-1/2} & \text{if $T$ is nonabelian,}
\end{array}
\right.
\end{equation}
where
\begin{equation}
\label{eq:f-star}
\alpha^*(T) = \left\{ 
\begin{array}{ll}
p         & \text{if $T = C_p$} \\
m/4       & \text{if $T = A_m$} \\
q^{r/30}  & \text{if $T$ is a group of Lie type of rank $r$ over $\F_q$} \\
2         & \text{if $T$ is a sporadic group or ${}^2F_4(2)'$.}
\end{array}
\right.
\end{equation}
To avoid ambiguity, we view $T$ as an alternating group if possible, we view $\POm^-_4(q)$ as $\PSL_2(q^2)$, $\PSL_3(2)$ as $\PSL_2(7)$ and $\PSp_4(3)$ as $\PSU_4(2)$.
Note that $\alpha(T) \to \infty$ as $\order{T} \to \infty$. 

\begin{lemma}
\label{lem:alpha_bound}
There exists a real number $C > 1$ such that all finite simple groups $T$ satisfy $\alpha(T) \geq C$.
\end{lemma}

\begin{proof}
If $T = C_p$ with $p$ prime, then
\[
\alpha(T) = p \geq 2.
\]
If $T = A_m$ with $m \geq 5$, then
\[
\alpha(T) = \left(\left(m/4\right)^{-2} + \tfrac{61}{60}\left(\half(m!/2)^{3/8}\right)^{-2}\right)^{-1/2} 
\geq \left(\left(\tfrac{5}{4}\right)^{-2} + \tfrac{61}{60}\left(\half\,60^{3/8}\right)^{-2}\right)^{-1/2} \geq 1.09.
\]
If $T$ is a group of Lie type of rank $r$ over $\F_q$, then $q \geq 3$ if $r = 2$ and $q \geq 7$ if $r = 1$, so
\[
\alpha(T) \geq \left((q^{r/30})^{-2} + \tfrac{61}{60}\left(\half\order{T}^{3/8}\right)^{-2}\right)^{-1/2} \geq 1.01.
\]
If $T$ is a sporadic group or ${}^2F_4(2)'$, then
\[
\alpha(T) = \left(2^{-2} + \tfrac{61}{60}\left(\half\order{T}^{3/8}\right)^{-2}\right)^{-1/2} 
\geq \left(2^{-2} + \tfrac{61}{60}\left(\half\,7920^{3/8}\right)^{-2}\right)^{-1/2} \geq 1.98.
\]
This proves that the result holds with $C = 1.01$.
\end{proof}

\section{Maximal subgroup zeta function}
\label{s:zeta}

The aim of this section is to prove Theorem~\ref{thm:zeta}, which concerns groups $G$ with a unique minimal normal subgroup $N$.
We will divide our analysis according to $N$.

\subsection{Preliminaries}
\label{ss:zeta_prelims}

For a group $G$ and a normal subgroup $N \normal G$, recall that $\M(G,N)$ is the set of maximal subgroups of $G$ that do not contain $N$ and let $\C(G,N)$ be a set of representatives for the conjugacy classes of members of~$\M(G,N)$.

\begin{lemma}
\label{lem:zeta_classes}
Let $G$ be a finite group with a unique minimal normal subgroup $N$.
Let $s$ be a real number.
Then 
\[
\sum_{M \in \M(G, N)} \order{G : M}^{-s} = \sum_{M \in \C(G, N)} \order{G : M}^{-(s-1)}.
\]
\end{lemma}

\begin{proof}
Let $M \in \M(G, N)$.
Since $N$ is the unique minimal normal subgroup of $G$ and $N \nleq M$, we know that $M$ is not normal in $G$, so $\Norm{G}{M}$ is a proper subgroup of $G$ containing $M$ and hence $\Norm{G}{M} = M$.
Therefore, $\order{M^G} = \order{G : M}$, which proves the result.
\end{proof}

\subsection{Groups with an abelian unique minimal normal subgroup}
\label{ss:zeta_abelian}

We begin with the case where $G$ is a group with an abelian unique minimal normal subgroup $N$.

\begin{proposition}
\label{prop:zeta_abelian}
Let $G$ be a finite group with a unique minimal normal subgroup $N$, and assume that $N$ is abelian.
Let $s$ be a real number.
Then 
\[
\sum_{M \in \C(G, N)} \order{G : M}^{-(s-1)} \leq \order{N}^{-(s-3/2)}.
\]
Equivalently, if $N = T^n$ where $T$ has prime order, then
\[
\sum_{M \in \C(G, N)} \order{G : M}^{-(s-1)} \leq \alpha(T)^{-n(s-3/2)}.
\]
\end{proposition}

\begin{proof}
Let $M \in \C(G, N)$. 
Then $M \cap N$ is normal in $G = MN$ since it is normalized by $M$ and $N$ is abelian. 
As $N \nleq M$ and $N$ is a minimal normal subgroup, we deduce that $M \cap N = 1$.  
Conversely, if $H$ is a complement to $N$ in $G$, then $H$ is contained in some conjugate of a member $M \in \C(G, N)$, but since $M$ is also a complement to $N$ in $G$, we deduce that $H$ is conjugate to $M$. 
Therefore, $\C(G, N)$ is a set of conjugacy class representatives for the complements to $N$ in $G$, which is well-known to be in bijection with the elements of the first cohomology group~$\Cohom{G/N,N}$ (see, for example, \cite[(11.1.3)]{Robinson96}). 
On the one hand, if $G/N$ does not act faithfully on $N$, then Lemma~\ref{lem:faithful} implies that $\Frat{G} \neq 1$ and hence $N \leq \Frat{G}$, so Lemma~\ref{lem:complemented} implies that $N$ has no complements, and we conclude that
\[
\sum_{M \in \C(G, N)} \order{G : M}^{-(s-1)} = 0.
\]
On the other hand, if $G/N$ does act faithfully on $N$, then $N$ is a faithful module for $G/N$ and it is irreducible since $N$ is a minimal normal subgroup of $G$, so \cite[Theorem~1]{GuralnickHoffman98} gives $\order{\Cohom{G/N,N}} \leq \order{N}^{1/2}$, and we conclude that
\[
\sum_{M \in \C(G, N)} \order{G : M}^{-(s-1)} = \frac{\order{\C(G, N)}}{\order{N}^{s-1}} \leq \order{N}^{-(s-3/2)},
\]
as required.
The final statement follows immediately from the definition of $\alpha(T)$.
\end{proof}

\subsection{Almost simple groups}
\label{ss:zeta_almost_simple}

Before considering the general case where $G$ is a group with a nonabelian unique minimal normal subgroup $N$, we focus on the special case where $N$ is a nonabelian simple group, so $G$ is almost simple, and we prove the following stronger version of our result in this case.

\begin{proposition}
\label{prop:zeta_almost_simple}
Let $G$ be an almost simple group with socle $T$.
Let $n \geq 1$ be an integer.
Then
\[
\sum_{M \in \C(G,T)} \frac{\order{\Out{T}}^{n-1}}{\order{G:M}^{n}} < \alpha^*(T)^{-n}.
\]
\end{proposition}

\begin{proof}
It suffices to prove that
\begin{equation} \label{eq:bound}
\sum_{M \in \C(G,T)} \order{G:M}^{-1} < \bigl(\alpha^*(T) \, \order{\Out{T}}\bigr)^{-1},
\end{equation}
since this implies that
\[
\sum_{M \in \C(G,T)} \frac{\order{\Out{T}}^{n-1}}{\order{G:M}^n} \leq
\order{\Out{T}}^{n-1} \biggl( \sum_{M \in \C(G,T)} \order{G:M}^{-1} \biggr)^{n}
< \alpha^*(T)^{-n},
\]
as required.
In most cases, we will show that \eqref{eq:bound} holds. 
Occasionally, it will be convenient to prove the weaker result that
\begin{equation} \label{eq:weak_bound}
\sum_{M \in \C(G,T)} \order{G:M}^{-1} < \alpha^*(T)^{-1} \qquad\text{and}\qquad \sum_{M \in \C(G,T)} \order{G:M}^{-2} < \bigl(\alpha^*(T) \, \order{\Out{T}}\bigr)^{-2},
\end{equation}
which is also easily seen to be sufficient.
We follow \cite{MenezesQuickRoneyDougal13}, where it is proved, amongst other things, that in most cases $P_2(G, T) \leq \sum_{M \in \C(G,T)} \order{G:M}^{-1} < 1/10$.

First assume that $T$ is a sporadic group or the Tits group ${}^2F_4(2)'$.
Then $\order{\Out{T}} \leq 2$ and \eqref{eq:bound} can be verified in \textsf{GAP} \cite{GAP} using the Character Table Library \cite{CTblLib}, as described in \cite[\S1]{code}.  

Next assume that $T$ is the alternating group $A_m$ for $m \geq 5$.
Inspecting the proof of \cite[Theorem~1.1]{MarotiTamburini11}, we see that if $m \geq 13$, then 
\[
\sum_{M \in \C(G,T)} \order{G:M}^{-1} < \frac{1}{m} + \frac{13}{m^2} \leq \frac{2}{m} = (\alpha^*(T) \, \order{\Out{T}})^{-1}.
\]
If $5 \leq m \leq 12$, then we use \textsc{Magma} \cite{Magma} to verify \eqref{eq:weak_bound} as explained in \cite[\S\S2.1--2.2]{code}, unless $G$ is $A_6$ or $S_6$.
If $G$ is $A_6$ or $S_6$, then 
\[
\sum_{M \in \C(G,T)} \frac{1}{\order{G:M}^{n}} = \frac{2}{6^n} + \frac{1}{10^n} + \frac{2}{15^n} < \frac{4}{6^n} = 4^{-(n-1)} \, \left(\frac{6}{4}\right)^{-n} = \order{\Out{T}}^{-(n-1)} \, \alpha^*(T)^{-n},
\]
which establishes the desired bound.

Now assume that $T$ is an exceptional group of Lie type over $\F_{q}$ where $q = p^f$ and $p$ is prime.
For the following low rank groups
\[
{}^2B_2(q) \ (q \geq 8), \quad {}^2G_2(q) \ (q \geq 27), \quad G_2(q) \ (q \geq 3), \quad {}^2F_4(q) \ (q \geq 8), \quad {}^3D_4(q) \ (q \geq 2),
\]
complete maximal subgroup information was available to \cite{MenezesQuickRoneyDougal13}, and the information in \cite[Table~2]{MenezesQuickRoneyDougal13} is sufficient to verify \eqref{eq:bound}.
For example, if $T = {}^2B_2(q)$, then, as in \cite[p.~376]{MenezesQuickRoneyDougal13}, $G$ has at most $\log{q}$ classes of maximal subfield subgroups, each of index exceeding $q^3$, and at most $4$ remaining classes of core-free maximal subgroups, each of index exceeding $q^2$, which gives
\[
\sum_{M \in \C(G,T)} \order{G:M}^{-1} < \frac{4}{q^2} + \frac{\log{q}}{q^3} \leq \frac{1}{q} \leq \frac{1}{q^{1/4}\log{q}} \leq \frac{1}{q^{1/4} f} \leq (\alpha^*(T) \, \order{\Out{T}})^{-1}.
\]
We adopt a similar approach for the medium rank groups $F_4(q)$ and $E^\pm_6(q)$, using the recent classification of maximal subgroups in \cite{Craven23}, which was not available to \cite{MenezesQuickRoneyDougal13}.
More precisely, it is sufficient to take the description of the maximal subgroups in the class $\K$ in \cite[Table~3]{MenezesQuickRoneyDougal13} and allow for an additional $8$ (respectively, $27$) classes of subgroups if $T = F_4(q)$ (respectively $T = E^\pm_6(q)$).
For example, if $T = E_6(q)$, then,
\[
\sum_{M \in \C(G,T)} \order{G:M}^{-1} < \frac{33 + 6\log{q} + 27}{q^{16}} < \frac{1}{q^{6/30} 6\log{q}} \leq (\alpha^*(T) \, \order{\Out{T}})^{-1}.
\]
For the large rank groups $E_7(q)$ and $E_8(q)$, we use the bounds in \cite[pp.~378 \&~380]{MenezesQuickRoneyDougal13}.
For example, if $T = E_8(q)$, then
\[
\sum_{M \in \C(G,T)} \order{G:M}^{-1} < \frac{50 + \log{q}}{q^{57}} + \frac{360q^{56}\log{q}\log{(12q^{56}\log{q})}}{7\order{\POm^+_{16}(q)}} < \frac{1}{q^{8/30}\log{q}} \leq (\alpha^*(T) \, \order{\Out{T}})^{-1}.
\]

Finally assume that $T$ is a classical group over $\F_{q}$ with natural module $V = \F_{q^u}^n$. 
For now consider the large rank groups
\[
\PSL_n(q) \ (n \geq 6), \quad \PSU_n(q) \ (n \geq 8), \quad \PSp_n(q) \ (\text{even $n \geq 8$}), \quad \POm^\epsilon_n(q) \ (n \geq 9).
\]
Following \cite[Section~4]{MenezesQuickRoneyDougal13}, we note that by \cite[Theorem~4.1]{Liebeck85}, each core-free maximal subgroup of $G$ is either one of at most
$m_1(G)$ known possibilities or has order strictly less than $q^{3un}$, where $u = 2$ if~$T$ is unitary and $u = 1$ otherwise, and in \cite[p.~386]{MenezesQuickRoneyDougal13} it is shown that
\[
m_1(G) \leq \left\{
\begin{array}{ll}
\tfrac{5}{2}n + 12n^{1/2} + 9 \log{n} + \log\log{q} + 12 & \text{if $T = \POm^+_n(q)$} \\
6n + \tfrac{1}{3} n\log{n} + n\log\log{q}                & \text{otherwise.}
\end{array}
\right.
\]
Moreover, by \cite[Theorem~1.1]{Hasa14}, the total number of maximal subgroups of $G$ is
\[
m(G) \leq 2n^{5.2} + n\log\log{q}.
\]
Therefore,
\[
\sum_{M \in \C(G,T)} \order{G:M}^{-1} < \frac{m_1(G)}{\rho(T)} + \frac{m(G)q^{3un}}{\order{T}},
\]
where $\rho(T)$ is the minimal index of a core-free subgroup of $T$.
Upper bounds on $\rho(T)$ are given in \cite[Table~5.2.A]{KleidmanLiebeck90}.
The authors of \cite{MenezesQuickRoneyDougal13} note that there are a couple of (unspecified) corrections to these bounds due to Bray, and consulting \cite[Table~2.7]{Menezes13} we see that these are
\[
\begin{array}{ll}
\rho(\PSU_n(2)) \leq \tfrac{1}{3}2^{n-1}(2^n-1)      & \text{if $n = 2m \geq 6$} \\[7pt]
\rho(\POm^+_{2m}(3)) \leq \tfrac{1}{2}3^{m-1}(3^m-1) & \text{if $n = 2m \geq 8$.}
\end{array}
\]
Combining these observations, it is easy to verify \eqref{eq:bound}, unless
$T$ is one of the following groups
\[
\PSL_6(2), \quad \PSU_8(2), \quad \PSp_8(2), \quad \PSp_8(3), \quad \Omega_9(3), \quad \POm^+_{10}(2), \quad \POm^-_{10}(2),
\]
in which case we use \textsc{Magma} (see \cite[\S2.3]{code}).

It remains to consider the low rank classical groups
\[
\PSL_n(q) \ (2 \leq n \leq 5), \quad \PSU_n(q) \ (3 \leq n \leq 7), \quad \PSp_n(q) \ (\text{even $4 \leq n \leq 6$}), \quad \POm_7(q), \quad \POm^\pm_8(q).
\]
Complete maximal subgroup information can be found in \cite{BrayHoltRoneyDougal13}, so these are dealt with just like the low rank exceptional groups. 
In particular, with the information in \cite[Table~4]{MenezesQuickRoneyDougal13} (and analogous bounds for $\PSL_5(q)$ and $\PSU_7(q)$ derived from the information in \cite[Tables~8.18, 8.19, 8.37 and~8.38]{BrayHoltRoneyDougal13} and \cite[Table~5.2.A]{KleidmanLiebeck90}), we verify \eqref{eq:bound} unless $T$ is one of the following groups
\begin{gather*}
\PSL_2(q) \ (7 \leq q \leq 8), \quad
\PSL_3(q) \ (3 \leq q \leq 16), \quad
\PSL_4(q) \ (3 \leq q \leq 5), \quad
\PSL_5(2), \\
\PSU_3(q) \ (3 \leq q \leq 8), \quad
\PSU_4(q) \ (q \leq 3), \quad
\PSU_5(2), \quad
\PSU_6(2), \quad
\PSp_6(2), \quad
\POm^+_8(q) \ (q \leq 3).
\end{gather*}
For these groups, we use \textsc{Magma} to verify \eqref{eq:bound}, except for when $T = \PSL_3(4)$, when we verify \eqref{eq:weak_bound} instead (see \cite[\S2.3]{code}).
\end{proof}

\subsection{Groups with a nonabelian unique minimal normal subgroup}
\label{ss:zeta_nonabelian}

We now consider the general case where $G$ is a group with a nonabelian unique minimal normal subgroup $N$.
This means that there is a nonabelian simple group $T$ such that
\[
N = T_{1} \times \dots \times T_{n}
\]
where $T_{i} \cong T$ for each $i$.  
Let us fix some notation that we will adopt for the rest of this section.
For each~$i$, let $\pi_{i} \colon N \to T$ be the projection map $(x_{1},\dots,x_{n}) \mapsto x_{i}$.  
Let $M \in \C(G, N)$, noting that $G = MN$.
Since $N$~is the minimal normal subgroup of~$G$, the action of~$G$ on~$N$ permutes the direct factors $T_{1}$, \dots,~$T_{n}$ transitively.  It follows that $M$~also permutes these factors transitively.  
Write $H_{i} = (M \cap N)\pi_{i}$, for each $i$, and view $H_{i}$ as a subgroup of the direct factor~$T_{i}$.

\begin{lemma}
\label{lem:MmeetN}
Adopt the notation above.
\begin{enumerate}
\item \label{i:M=Norm}
If $M \cap N \neq 1$, then $M = \Norm{G}{M \cap N}$.
\item \label{i:ProjectionsConjugate}
If $g \in M$ satisfies $T_{i}^{g} = T_{j}$ for some $i$~and~$j$, then $H_{i}^{g} = H_{j}$.  
In particular, the subgroups~$H_{i}$ are images of each other under automorphisms of~$T$.
\end{enumerate}
\end{lemma}

\begin{proof} \
\begin{enumerate}
\item
Suppose that $M \cap N \neq 1$.  
Since $G = MN$, certainly $N \nleq M$.  
Therefore, $1 < M \cap N < N$, so $\Norm{G}{M \cap N}$~is a proper subgroup of~$G$ that contains the maximal subgroup~$M$.  
Therefore, $M = \Norm{G}{M \cap N}$.

\item 
Since $g$~induces an automorphism of~$N$, there exists a permutation~$\sigma$ in~$\Sym{n}$ and automorphisms $\alpha_{1}$, \dots,~$\alpha_{n}$ of~$T$ such that if $(x_{1},\dots,x_{n}) \in N$ then
\[
(x_{1},\dots,x_{n})^{g} =
(x_{1\sigma^{-1}}^{\alpha_{1\sigma^{-1}}}, \dots, x_{n\sigma^{-1}}^{\alpha_{n\sigma^{-1}}}).
\]
In particular, $i\sigma = j$.  
Let $x = (x_{1},\dots,x_{n})$ be an arbitrary element of~$M \cap N$.  
Then $x^{g}$~is also in~$M \cap N$ and
\[
(x^{g})\pi_{j} = (1,\dots,1,x_{i}^{\alpha_{i}},1,\dots,1)
\]
where the element~$x_{i}^{\alpha_{i}}$ appears in the $j$th coordinate.  
The latter is also the image of the element $(1,\dots,1,x_{i},1,\dots,1)$ under conjugation by~$g$.  
Hence
\[
H_{i}^{g} = \bigl( (M \cap N)^{g} \bigr)\pi_{j} = (M \cap N)\pi_{j} = H_{j},
\]
as claimed.
\end{enumerate}
\end{proof}

\begin{proposition}
\label{prop:zeta_nonabelian}
Let $G$ be a finite group with a unique minimal normal subgroup $N \cong T^{n}$ where $T$~is a nonabelian simple group.
Let $d \geq 2$ be an integer.
Then
\[
\sum_{M \in \C(G, N)} \order{G : M}^{-(d-1)} \leq \alpha(T)^{-n(d-1)}.
\]
\end{proposition}

\begin{proof}
Adopt the notation from above.
For each $M \in \C(G, N)$, by Lemma~\ref{lem:MmeetN}\ref{i:ProjectionsConjugate}, the subgroups $H_{i} = (M \cap N)\pi_{i}$ are images of each other under automorphisms of~$T$, yielding three possibilities:
\begin{enumerate}[(1),itemsep=0pt]
\item $H_{i} = T$ for each~$i$
\item $1 < H_{i} < T$ for each~$i$
\item $H_{i} = 1$ for each~$i$.
\end{enumerate}
For $i \in \{1,2,3\}$, write~$\C_{i}$ for the set of subgroups $M \in \C(G, N)$ that satisfy the condition~($i$) above, so $\C(G, N) = \C_{1} \cup \C_{2} \cup \C_{3}$.  
Note that $\C_{1}$ and $\C_{3}$ are empty when $n = 1$.
We shall consider each case below.  
Cases~1 and~3 are quicker to handle, so we present them first.
In these two cases, our proof closely follows the arguments in \cite[Lemmas~3.5 and~3.6]{DetomiLucchini13}.
Case~2 requires a different argument and, here, we will rely heavily on the bound obtained in Proposition~\ref{prop:zeta_almost_simple}.

\paragraph{Case~1} 
  If $M \in \C_{1}$, then $M \cap N$~is a
  subdirect product of copies of the nonabelian simple group~$T$.
  There is therefore a partition of $\Delta = \{1,\dots,n\}$ into
  disjoint subsets $\Delta = \Delta_{1} \cup \dots
  \cup \Delta_{k}$ such that
  \[
  M \cap N = D_{1} \times \dots \times D_{k}
  \]
  where each~$D_{j}$ is a diagonal subgroup of the product $\prod_{i
    \in \Delta_{j}} T_{i}$.  These diagonal subgroups~$D_{j}$ are the
  minimal normal subgroups of~$M \cap N$ and hence they are permuted
  by~$M$ in its conjugation action on~$M \cap N$.  It follows that
  each~$\Delta_{j}$~is a block for the induced action~$M$ on the
  set~$\Delta$ and therefore also for~$G = MN$.  Thus the partition
  $\{ \Delta_{1}, \dots, \Delta_{k} \}$ is a block system
  for the action of~$G$ on~$\Delta$.  Observe also that
  \[
  \order{G:M} = \order{MN:M} = \order{N: M \cap N} = \order{T}^{n-k}.
  \]
  We now bound the number of conjugacy classes of maximal
  subgroup~$M$ arising in this form.

  Fix some partition of~$\Delta$ into blocks for the action of~$G$;
  say, $\Delta = \Delta_{1} \cup \dots \cup
  \Delta_{k}$.  The diagonal subgroup~$D_{j}$ corresponding to the
  block~$\Delta_{j}$ has the form
  \[
  D_{j} = \set{ (x,x^{\beta_{2}},\dots,x^{\beta_{r}}) }{ x \in T }
  \]
  (identified with a subgroup of~$T^{\Delta_{j}}$), where $\beta_{2}$,
  \dots,~$\beta_{r}$ are automorphisms of~$T$.  Note that if
  $\beta_{i} \equiv \beta'_{i} \pmod{\Inn{T}}$ for $i = 2$,
  \dots,~$r$, then the corresponding diagonal subgroups
  $D_{j}$~and~$D'_{j}$ are conjugate via some element of~$N$.  Hence,
  the number of conjugacy classes of subgroups of~$G$ of the form
  \begin{equation}
    D_{1} \times \dots \times D_{k},
    \label{eq:Case1-intersection}
  \end{equation}
  for this specific block system, is bounded above
  by~$\order{\Out{T}}^{n-k}$.  Since $M = \Norm{G}{M \cap N}$ is
  determined by the intersection~$M \cap N$, the same bound applies to
  the number of conjugacy classes of maximal subgroups~$M$ that have
  intersection~$M \cap N$ of the form~\eqref{eq:Case1-intersection}
  for the specified partition of~$\Delta$.  Let us
  write~$\C_{1}(\Delta_{1})$ for the members of~$\C_{1}$ that arise
  from a block system with~$\Delta_{1}$ as one of the blocks.  Then,
  as $\order{G:M} = \order{T}^{n-k}$,
  \[
  \sum_{M \in \C_{1}(\Delta_{1})} \order{G:M}^{-(d-1)} \leq
  \left( \frac{\order{\Out{T}}}{\order{T}^{d-1}} \right)^{n-k} \leq
  \left( \frac{\order{\Out{T}}}{\order{T}^{d-1}} \right)^{n-2}.
  \]
  The number of choices for the block~$\Delta_{1}$ can be bounded
  above by~$2^{n}$ (that is, the number of subsets of~$\Delta$) and
  this block then determines the block system.  We therefore conclude
  that
  \[
  \sum_{M \in \C_{1}} \order{G:M}^{-(d-1)} \leq
  2^{n} \left( \frac{\order{\Out{T}}}{\order{T}^{d-1}} \right)^{n/2} =
  \left( \frac{4\,\order{\Out{T}}}{\order{T}^{d-1}} \right)^{n/2}.
  \]

\paragraph{Case~3} 
  Observe that $M \in \C_{3}$ precisely when
  $M$~is a maximal subgroup of~$G$ that complements the minimal normal
  subgroup~$N$.  Assume that there is at least one such $M \in
  \C_{3}$.  Then the core of~$M$ in~$G$ is trivial, so $G$~can be
  viewed as a primitive permutation group via its action on the cosets
  of~$M$.  According to the O'Nan--Scott Theorem, $G$~is isomorphic to
  a twisted wreath product $G \cong T \wr_{\Norm{G}{T_{1}},\phi} M$
  where the homomorphism $\phi \colon \Norm{G}{T_{1}} \to \Aut T$,
  arising via the conjugation action on the first factor~$T_{1}$
  of~$N$, has the property that $\Inn{T} \leq \im\phi$ and where
  $M$~acts faithfully and transitively on the collection
  $\{T_{1},\dots,T_{n}\}$ of factors of~$N$ (see Case~2(a) of
  the proof of the Main Theorem of~\cite{LiebeckPraegerSaxl88}).  In
  particular, $\order{G/N} \leq n!$ and each complement $M \in \C_{3}$
  is large in the sense of~\cite{JaikinZapirainPyber11}.  We
  apply~\cite[Proposition~2.16]{JaikinZapirainPyber11} to tell us that
  \[
  \order{\C_{3}} \leq \order{\Out{T}} \log {\order{G/N}} \leq
  n^{2} \, \order{\Out{T}}.
  \]
  Since $\order{G:M} = \order{N}$ for each $M \in \C_{3}$, we conclude
  that
  \[
  \sum_{M \in \C_{3}} \order{G:M}^{-(d-1)} \leq
  \frac{n^{2} \, \order{\Out{T}}}{\order{T}^{n(d-1)}}.
  \]

\paragraph{Case~2} 
  We first define~$A$ to be the image of the normalizer~$\Norm{G}{T_{1}}$ in $\Aut{T}$.  
  Thus $A$~is an almost simple group with socle~$T$ that is isomorphic to~$\Norm{G}{T_{1}}/\Cent{G}{T_{1}}$.
  Note that $A$~is fixed and is independent of any choice of maximal subgroup~$M$. 
  
  Now let $M \in \C_{2}$.  
  Our argument draws on, and runs parallel with, the proof of \cite[Theorem~4.6A]{DixonMortimer96}.  
  Write $H = H_{1} \times \dots \times H_{n}$.  
  By Lemma~\ref{lem:MmeetN}\ref{i:ProjectionsConjugate}, $M$~normalizes~$H$.  
  Since $M$~is maximal in~$G$, either $G = HM$ or $H \leq M$.  
  If it were the case that $G = HM$, then $G$~would also normalize~$H$, but this is a contradiction as $1 < H < N$ and $N$~is minimal normal. 
  Therefore, $H \leq M$, and it follows that $M \cap N = H = H_{1} \times \dots \times H_{n}$.

  We shall view~$G$ as a primitive group via its action on the
  set~$\Omega$ of cosets of the maximal subgroup~$M$.  We fix~$\omega
  \in \Omega$ such that the stabilizer~$G_{\omega}$ equals~$M$.  We
  shall also let $T$~act on the set~$\Delta$ of cosets of~$H_{1}$
  in~$T$ and so view $T$~as a subgroup of~$\Sym{\Delta}$.  We also fix
  $\delta \in \Delta$ such that the stabilizer~$T_{\delta}$
  equals~$H_{1}$.  We then let $T^{n}$~act on~$\Delta^{n}$ via the
  product action and hence $T^{n} \leq \Sym{\Delta^{n}}$.  Since the
  subgroups~$H_{i}$ are images of each other under automorphisms
  of~$T$, there is an isomorphism $N \to T^{n}$ which maps the
  subgroup $M \cap N = H_{1} \times \dots \times H_{n}$
  to~$H_{1}^{n}$.  As the latter subgroups are the stabilizers for the
  relevant actions, this therefore defines a permutation isomorphism
  from $N \leq \Sym{\Omega}$ to $T^{n} \leq
  \Sym{\Delta^{n}}$.  Hence, there is a bijection $\phi \colon \Omega
  \to \Delta^{n}$ such that this permutation isomorphism is given by
  $\phi^{\ast} \colon x \mapsto \phi^{-1}x\phi$ for $x \in T^{n}$.
  Note that $\phi$~maps our chosen point~$\omega \in \Omega$ to the
  sequence $(\delta,\dots,\delta)$
  in~$\Delta^{n}$.  Now $\phi^{\ast}$~is also an isomorphism
  $\Sym{\Omega} \to \Sym{\Delta^{n}}$ which maps the normalizer of~$N$
  in~$\Sym{\Omega}$ to the normalizer of~$T^{n}$
  in~$\Sym{\Delta^{n}}$.  The latter is equal to the wreath product $W
  = \Norm{\Sym{\Delta}}{T} \wr S_{n}$ (as is noted in
  \cite[Lemma~4.5A]{DixonMortimer96}).  Hence, $\phi^{\ast}$~maps~$G$
  to a primitive subgroup of~$W$.  Consequently, $W$~is also primitive
  in its action on~$\Delta^{n}$ and therefore the normalizer
  $\Norm{\Sym{\Delta}}{T}$ is primitive in its action on~$\Delta$.  In
  particular, \cite[Theorem~4.3B]{DixonMortimer96} tells us that
  $\Norm{\Sym{\Delta}}{T}$~is almost simple with socle~$T$.

  Let us now consider the image of the normalizer~$\Norm{G}{T_{1}}$
  under~$\phi^{\ast}$.  For simplicity of notation, we also
  write~$T_{1}$ for the first factor of~$T^{n}$ occurring as a
  subgroup of the wreath product~$W$.  Observe that $\Norm{W}{T_{1}}$
  decomposes as a direct product $\Norm{\Sym{\Delta}}{T} \times (
  \Norm{\Sym{\Delta}}{T} \wr S_{n-1})$ and hence there is a projection
  map $\mu \colon \Norm{W}{T_{1}} \to \Norm{\Sym{\Delta}}{T_{1}}$.
  Since $\Norm{\Sym{\Delta}}{T}$~is almost simple with socle~$T$, one
  deduces $\ker\mu = \Cent{W}{T_{1}}$ and hence
  $\Norm{G}{T_{1}}\phi^{\ast} \cap \ker\mu =
  \Cent{G}{T_{1}}{\phi^{\ast}}$.  It follows that
  $\phi^{\ast}\mu$~induces an isomorphism~$\alpha$ from $A =
  \Norm{G}{T_{1}}/\Cent{G}{T_{1}}$ to the image $B =
  \Norm{G}{T_{1}}\phi^{\ast}\mu$ in~$\Norm{\Sym{\Delta}}{T}$.
  We now
  apply~\cite[(2.2)]{Kovacs89} to produce an element $z \in
  (\Norm{\Sym{\Delta}}{T})^{n}$ such that $(G\phi^{\ast})^{z} \leq B
  \wr S_{n}$.  We deduce that the latter wreath product is also
  primitive in its action on~$\Delta^{n}$ and hence the action of~$B$
  on~$\Delta$ is primitive (again by
  \cite[Lemma~4.5A]{DixonMortimer96}).  Therefore, the
  stabilizer~$B_{\delta}$ is a maximal subgroup of~$B$.  By
  construction, $\phi^{\ast}$~maps $H_{1}$~and~$T_{1}$ occurring as
  subgroups of~$G$ to the corresponding subgroups of~$W$.  Hence
  $(T_{1}\Cent{G}{T_{1}})\phi^{\ast}\mu = T$ and
  $(H_{1}\Cent{G}{T_{1}})\phi^{\ast}\mu = H_{1}$.  Therefore, the
  isomorphism $\alpha \colon A \to B$ maps the subgroups
  $H_{1}$~and~$T$ in~$A$ to the corresponding subgroups of~$B$.  Since
  $T \cap B_{\delta} = T_{\delta} = H_{1}$, we conclude that $L =
  B_{\delta}\alpha^{-1}$ is a maximal subgroup of~$A$ such that $L
  \cap T = H_{1}$.

  We are now able to completely describe the maximal subgroup~$M$.  As
  noted in Lemma~\ref{lem:MmeetN}\ref{i:ProjectionsConjugate}, the
  other images $H_{2}$, \dots,~$H_{n}$ are images of~$H_{1}$ under
  automorphisms of~$T$; say, $H_{i} = H_{1}^{\alpha_{i}}$ where
  $\alpha_{i} \in \Aut{T}$.  Therefore, the intersection $M \cap N =
  H_{1} \times H_{2} \times \dots \times H_{n}$ is determined
  by~$H_{1}$ and these~$\alpha_{i}$ and, as $M = \Norm{G}{M \cap N}$,
  this then determines the maximal subgroup~$M$.

  Now suppose that $R$~is another maximal subgroup of~$G$ such that
  $J_{1} = (R \cap N)\pi_{1} = L^{x} \cap T$ for some $x \in A$.  Then
  $1 < J_{1} < T$, so by our above argument $R = \Norm{G}{J_{1}
    \times J_{2} \times \dots \times J_{n}}$ where each~$J_{i}$ is the
  image of~$J_{1}$ under some automorphism of~$T$.  We continue to use
  the map~$\phi^{\ast}$ determined above by the choice of maximal
  subgroup~$M$ and its related technology.  In particular, $x\alpha
  \in B$, so there exists $g \in G$ such that $g\phi^{\ast}\mu = x$.
  Consider the conjugate~$M^{g}$ of our maximal subgroup~$M$ by this
  element~$g$.  By construction, $(M \cap N)\phi^{\ast}\mu =
  H_{1}^{n}\mu = H_{1}$ and hence $(M^{g} \cap N)\phi^{\ast}\mu =
  H_{1}^{x} = J_{1}$.  Therefore, by the previous paragraph, there
  exist $\beta_{2},\dots,\beta_{n} \in \Aut{T}$ such that $M^{g} \cap
  N = J_{1} \times J_{2}^{\beta_{2}} \times \dots \times
  J_{n}^{\beta_{n}}$.  Thus the maximal subgroup~$R$ differs from the
  conjugate~$M^{g}$ by conjugation by a selection of
  $n-1$~automorphisms of~$T$.

  In conclusion, taking into account conjugation by elements of~$N$,
  the members $M \in \C_{2}$ are determined by a choice of
  representative for the conjugacy classes of maximal
  subgroups~$L$ of the almost simple group~$A$ (necessarily satisfying $T \nleq L$) and
  $n-1$~representatives in the outer automorphism group of~$T$.  
  Noting that $H_{1} = L \cap T$, we obtain
  \[
  \order{G:M} = \order{N:M \cap N} = \order{T:L \cap T}^{n} =
  \order{A:L}^{n}.
  \]
  Hence, we conclude that
  \[
  \sum_{M \in \C_{2}} \order{G:M}^{-(d-1)} \leq
  \sum_{L \in \C(A, T)} 
  \frac{\order{\Out{T}}^{n-1}}{\order{A:L}^{n(d-1)}}.
  \]
  Therefore, by Proposition~\ref{prop:zeta_almost_simple},
  \[
  \sum_{M \in \C_{2}} \order{G:M}^{-(d-1)} \leq \alpha^*(T)^{-n(d-1)}.
  \]
  
  We now bring the three cases together. 
  If $n = 1$, then
  \[
  \sum_{M \in \C(G,N)} \order{G:M}^{-(d-1)} \leq \alpha^*(T)^{-n(d-1)} \leq \alpha(T)^{-n(d-1)},
  \]
  as required.
  It remains to assume that $n \geq 2$, in which case,
  \[
  \sum_{M \in \C(G, N)} \order{G:M}^{-(d-1)}
  \leq \left( \frac{4\,\order{\Out{T}}}{\order{T}^{d-1}} \right)^{n/2}
  + \frac{n^{2} \, \order{\Out{T}}}{\order{T}^{n(d-1)}}
  + \alpha(T)^{-n(d-1)}.
  \]
  A similar calculation shows $n^{2} \leq (2\sqrt{60})^{n}/60 \leq (4\order{T})^{n(d-1)/2}/60$ and, by~\cite[Proposition~4.4]{LiebeckPyberShalev07}, $\order{\Out{T}} \leq \order{T}^{1/4}$.  Hence
  \[
  \frac{n^{2} \, \order{\Out{T}}}{\order{T}^{n(d-1)}} 
  \leq \tfrac{1}{60} \left( \frac{4\,\order{\Out{T}}}{\order{T}^{d-1}} \right)^{n/2} 
  \leq \tfrac{1}{60} \left( \frac{4\,\order{T}^{1/4}}{\order{T}^{d-1}} \right)^{n/2} 
  \leq \tfrac{1}{60} \left( \tfrac{1}{2} \order{T}^{3/8} \right)^{-n(d-1)}.
  \]
  Therefore, noting that $n(d-1) \geq 2$, we conclude
  \[
  \sum_{M \in \C(G, N)} \order{G:M}^{-(d-1)}
  \leq \alpha^*(T)^{-n(d-1)} + \tfrac{61}{60} \left( \tfrac{1}{2} \order{T}^{3/8} \right)^{-n(d-1)}
  \leq \alpha(T)^{-n(d-1)},
  \]
  as required.
\end{proof}

Theorem~\ref{thm:zeta} now follows by combining Lemma~\ref{lem:zeta_classes} with Propositions~\ref{prop:zeta_abelian} and~\ref{prop:zeta_nonabelian}.

\section{Groups with prescribed chief tail}
\label{s:proof}

\subsection{Widths of chief factors}
\label{ss:width}

The aim of this section is to prove the following result on the growth of the widths of chief factors in a prescribed chief tail.

\begin{theorem}
\label{thm:width}
Let $G$~be a finite group with a prescribed chief tail $\T = (N_0, \dots, N_\ell)$.
Assume that $\Frat{G} = 1$ and $\Frat{G/N_j} = 1$ for some $1 \leq j \leq \ell-1$, and let $j$ be maximal subject to this condition.
Write $N_{j-1}/N_{j} = S^{m}$ and $N_{\ell-1} = T^{n}$ where $S$~and~$T$ are simple.
\begin{enumerate}
\item 
\label{it:width_nn}
If $S$ and $T$ are nonabelian, then $n \geq pm$ where $p$ is the greatest prime divisor of $\order{S}$.
\item 
\label{it:width_an}
If $S = C_p$ and $T$ is nonabelian, then $n \geq \delta m$ where
\[
\delta = 
\begin{cases}
\frac{1}{2} & \text{if $p = 2$ and $T = {\rm P}\Omega_8^+(q)$ for some odd prime power~$q$,} \\
$1$         & \text{otherwise.}
\end{cases}
\]
\item 
\label{it:width_na}
If $S$ is nonabelian and $T = C_p$, then $n \geq tm$ where $t$ is smallest degree of a faithful projective representation of $S$ in positive characteristic.
\item 
\label{it:width_aa}
If $S$ and $T$ are abelian, then $n \geq m$ and, if $n \geq 5$, then $n \geq m + 1$.
\end{enumerate}
\end{theorem}

\begin{remark}
\label{rem:width}
There are two cases in Theorem~\ref{thm:width} where we do not assert that $n > m$, and Example~\ref{ex:width} shows that these are genuine exceptions.
Indeed, for part~\ref{it:width_an}, Example~\ref{ex:width}\ref{it:width_almost_simple} shows that it is possible to have $n = m/2$ when $p = 2$ and $T = \POm_8^+(q)$ for odd $q$.
Similarly, for part~\ref{it:width_aa}, Example~\ref{ex:width}\ref{it:width_affine} shows that it is possible to have $n = m = 4$.
\end{remark}

Before proving Theorem~\ref{thm:width} we will need a number of preliminary results.
We begin by recording a well-known lemma.

\begin{lemma}
\label{lem:no_normal_p-group}
Let $n$ be a positive integer, let $F$ be a field of positive characteristic $p$, let $G$ be a finite group and let $\rho \colon G \to \GL_n(F)$ be a faithful irreducible representation.
Then $\Op{G} = 1$.
\end{lemma}

\begin{proof}
Since $\rho$ is irreducible, Clifford's Theorem implies that the restriction $\rho|_{\Op{G}}$ is completely reducible.
The only irreducible representations of a $p$-group in characteristic $p$ are trivial, so $\rho|_{\Op{G}}$ is trivial, but $\rho$ is faithful, so $\Op{G} = 1$.
\end{proof}

\begin{lemma}
\label{lem:rank_out}
Let $T$ be a nonabelian finite simple group. 
\begin{enumerate}
\item 
If $T \cong \POm^+_8(p^f)$ where $p$ is an odd prime number and $f \geq 1$, then $\Out{T}$ has a unique noncyclic chief factor and it is isomorphic to $C_2 \times C_2$.
\item 
Otherwise, every chief factor of\/ $\Out{T}$ is cyclic.
\end{enumerate}
\end{lemma}

\begin{proof}
The result certainly holds when $\Out{T}$ is abelian, so assume otherwise.
In particular, $T$ is a group of Lie type over $\F_{p^f}$ for a prime number $p$ and an integer $f \geq 1$.
The structure of $\Out{T}$ is described in \cite[Theorem~2.5.12]{GorensteinLyonsSolomon98}, and from this description one sees that, since $\Out{T}$ is nonabelian, $T$ is isomorphic to one of
\begin{enumerate}[(a),itemsep=0pt]
\item 
\label{it:out_A}
$\PSL^\epsilon_n(p^f)$ with $n \geq 3$ 
\item 
\label{it:out_D}
$\POm^\epsilon_n(p^f)$ with $n \geq 8$ even 
\item 
\label{it:out_E}
$E_6^\epsilon(p^f)$ with $p^f \equiv \epsilon1 \pmod{3}$. 
\end{enumerate}
In each case, one can determine the precise structure of $\Out{T}$ from \cite[Theorem~2.5.12]{GorensteinLyonsSolomon98}, but for convenience we give more explicit references.
In \ref{it:out_A}, $\Out{T} \cong C_{\gcd(p^f-1,n)}\semidirect(C_f \times C_2)$ if $\epsilon = +$ and $\Out{T} \cong C_{\gcd(p^f+1,n)}\semidirect C_{2f}$ if $\epsilon = -$ (see also \cite[Propositions~2.2.3 and~2.3.5]{KleidmanLiebeck90}), and in both cases $\Out{T}$ has a normal series with cyclic factors, so all chief factors of $\Out{T}$ are cyclic.
In \ref{it:out_E}, $\Out{T} \cong S_3 \times C_f$ (see also \cite[Lemmas~3.10 and~3.12]{BurnessGuralnickHarper21}), so, again, every chief factor of $\Out{T}$ is cyclic.
Finally consider \ref{it:out_D}. 
If $T \not\cong \POm_8^+(p^f)$, then, since $\Out{T}$ is nonabelian, $\Out{T} \cong D_8 \times C_f$ (see also \cite[Propositions~2.7.3 and~2.8.2]{KleidmanLiebeck90}) all of whose chief factors are cyclic.
If $T \cong \POm_8^+(p^f)$, then $\Out{T} \cong S_3 \times C_f$ if $p = 2$, and $\Out{T} \cong S_4 \times C_f$ if $p > 2$ (see \cite[Remark on p.38]{KleidmanLiebeck90}), from which the result follows.
\end{proof}

\begin{lemma}
\label{lem:frattini_chief_factors}
Let $G$ be a finite group with a prescribed chief tail $\T = (N_0, \dots, N_\ell)$.
Assume that $\Frat{G/N_j} = 1$ for some $1 \leq j \leq \ell$, and let $j$ be maximal subject to this condition.
Then $\Frat{G} = N_j$.  
Furthermore, if $N_{j}$~is nontrivial, then there is a prime~$p$ such that $N_{j} \leq \Op{G} \leq N_{j-1}$.
\end{lemma}

\begin{proof}
On the one hand, if $\Frat{G} \in \T$, then, since $\Frat{G/\Frat{G}} = 1$, the maximality of $j$ forces $N_{j} \leq \Frat{G}$. 
On the other hand, if $\Frat{G} \notin \T$, then, by definition, $N_{j} < N_{0} \leq \Frat{G}$. 
Therefore, in either case, $N_{j} \leq \Frat{G}$.  
As such, $\Frat{G/N_{j}} = \Frat{G}/N_{j}$, but, by construction, $\Frat{G/N_{j}} = 1$, so $\Frat{G} = N_{j}$.
In particular, $N_{j}$~is nilpotent and hence a direct product of its Sylow subgroups, each of which is therefore normal in~$G$.
Since $\T$~is the prescribed chief tail of~$G$, we conclude that $N_{j}$~is a (possibly trivial) $p$\nbd group.

Assume that $N_{j}$ is nontrivial.  
Certainly then $N_{j} \leq \Op{G}$.  
If $N_{j-1}$~is not a $p$\nbd group then, again using the fact that $\T$~is the tail of any chief series, $\Op{G} = N_{j}$, which would establish the result. 
Assume then that $N_{j-1}/N_{j} = C_{p}^{n}$ for some~$n$.  
Since $\Frat{G/N_{j}} = 1$, Lemma~\ref{lem:faithful} tells us that $\Cent{G/N_{j-1}}{N_{j-1}/N_{j}} = 1$, so $G/N_{j-1}$ embeds in $\Aut{(N_{j-1}/N_{j})} = \GL_n(p)$.  
Moreover, since $N_{j-1}/N_{j}$ is a minimal normal subgroup, $G/N_{j-1}$ embeds as an irreducible subgroup of $\GL_(p)$.  
By Lemma~\ref{lem:no_normal_p-group}, $\Op{G/N_{j-1}} = 1$, which forces $\Op{G} \leq N_{j-1}$, as required.
\end{proof}

The following is based on the proof of \cite[Theorem~1]{HarperLiebeck25}, which, in turn, is inspired by the proof of \cite[Theorem~(II')]{FeitTits78}.
Here, and throughout, a projective representation $\lambda \colon H \to \PGL(V)$ is \emph{imprimitive} if the preimage of $H$ in $\GL(V)$ permutes the subspaces $V_1$, \dots, $V_m$ in a direct sum decomposition $V = \bigoplus_{i=1}^m V_i$ where $m>1$, and \emph{primitive} otherwise.

\begin{lemma}
\label{lem:feit-tits}
Let $n$ and $r$ be positive integers and let $p$ and $q$ be prime numbers.
Let $H$ be a finite group and let $\lambda \colon H \to \PGL_n(\FF_p)$ be a faithful primitive projective representation.
Let $N$ be a normal subgroup of $H$ of order $q^r$.
Then $r \leq n$.
Moreover, if $n > 4$, then $r < n$.
\end{lemma}

\begin{proof}
Let $\eta \colon \widetilde{H} \to H$ be a finite central extension such that $\lambda$ lifts to a faithful representation $\widetilde{\lambda} \colon \widetilde{H} \to \GL_n(\FF_p)$, and, if $p \neq 2$, then choose $\widetilde{H}$ such that $\widetilde{H}\widetilde{\lambda}$ contains the subgroup $Z_4$ of $\Centre{\GL_n(\FF_p)}$ of order $4$. 
Let $\widetilde{N}$ be the full preimage of $N$ in $\widetilde{H}$.
Note that $\widetilde{N}$, being a central extension of a $q$-group, is nilpotent.

Let $\widetilde{M} \leq \widetilde{N}$ with $\widetilde{M} \normal \widetilde{H}$. 
As $\widetilde{\lambda}$ is irreducible, by Clifford's Theorem, $\widetilde{\lambda}|_{\widetilde{M}} = \alpha_1 \oplus \cdots \oplus \alpha_k$ where $\alpha_1, \dots, \alpha_k$ are the homogeneous components of $\widetilde{\lambda}|_{\widetilde{M}}$ whose respective irreducible components are pairwise nonisomorphic. 
Moreover, $\widetilde{H}$ acts transitively on these $k$ components, so $k = 1$ as $\widetilde{\lambda}$ is primitive.
Thus $\widetilde{\lambda}|_{\widetilde{M}} = \alpha_1$, the direct sum of $n_1$ isomorphic irreducible representations of dimension $n_2$. 
Therefore, $\lambda = \lambda_1 \otimes \lambda_2$ where $\lambda_1 \colon H \to \PGL_{n_1}(\FF_p)$ and $\lambda_2 \colon H \to \PGL_{n_2}(\FF_p)$ are irreducible projective representations of $H$ (see \cite[Theorem~3]{Clifford37}). 
Moreover, $\lambda_1$ and $\lambda_2$ are primitive since $\lambda$ is.

For now assume that there exists a choice of $\widetilde{M}$ such that $1 < n_1 < n$ and $1 < n_2 < n$. 
If $i \in \{1, 2\}$, then $N\lambda_i$ is a normal $q$-subgroup of the primitive subgroup $H\lambda_i \leq \PGL_{n_i}(\FF_p)$, so, by induction, $\order{N\lambda_i} \leq q^{n_i}$.
Since $\lambda = \lambda_1 \otimes \lambda_2$ is faithful, the intersection $\ker\lambda_1 \cap \ker\lambda_2$ is trivial, so $\order{N} \leq \order{N\lambda_1} \cdot \order{N\lambda_2} \leq q^{n_1+n_2}$.
In all cases $r \leq n_1 + n_2 \leq n_1n_2 = n$, and if $n > 4$, then $r \leq n_1 + n_2 < n_1n_2 = n$, as required.

It remains to assume that for all choices of $\widetilde{M}$, we have $n_1 = 1$ or $n_2 = 1$, which is to say, either $\widetilde{\lambda}_{\widetilde{M}}$ is irreducible, or $\widetilde{\lambda}|_{\widetilde{M}}$ a direct sum of isomorphic linear representations. 
In the latter case, $\widetilde{M}$ is represented by scalars, so $\widetilde{M}$ is a cyclic subgroup of $\Centre{\GL_n(\FF_p)}$.  
In particular, if $\widetilde{\lambda}|_{\widetilde{N}}$~is reducible, then $\widetilde{N} \leq \Centre{\GL_n(\FF_p)}$ and $N = \widetilde{N}\eta = 1$, which establishes the required result.

We may therefore assume that $\widetilde{\lambda}|_{\widetilde{N}}$ is irreducible.
In this case, the hypotheses of \cite[Proposition~2.4]{HarperLiebeck25} are satisfied by $\widetilde{N}$ and, by inspecting the proof of this proposition (see also \cite[(3.4)]{FeitTits78}), we see that $n = q^m$ and $N = \widetilde{N}\eta$ is the elementary abelian group of order $q^{2m}$.
In particular, $r = 2m$. 
Therefore, $r \leq q^m = n$, and, if $n > 4$, then $r < q^m = n$, as required.
\end{proof}

\begin{proposition}
\label{prop:abelian_abelian}
Let $n$ and $s$ be positive integers and let $p$ and $q$ be prime numbers. 
Let $H$ be an irreducible subgroup of $\GL_n(p)$ and let $\T$ be a prescribed chief tail of $H$.
Let $N \in \T$ be a nontrivial $q$-group and let $N/M$ be a chief factor of $H$ of order $q^s$.
Then $s \leq n$.
Moreover, if $n > 4$, then $s < n$.
\end{proposition}

\begin{proof}
Since $H \leq \GL_n(p)$ is irreducible and $N$ is a nontrivial $q$-group, Lemma~\ref{lem:no_normal_p-group} forces $q \neq p$.
Since $H \leq \GL_n(p)$ is irreducible, there is a divisor $m$ of $n$ such that $H \leq \GL_m(p^{n/m})$ is absolutely irreducible, so $H \leq \GL_m(\FF_p)$ is irreducible (see \cite[Lemma~2.10.2]{KleidmanLiebeck90}).
Write $V = \FF_p^m$.
If $m = 1$, then $\GL_m(p^{n/m})$ is cyclic, so $H$ is cyclic and the result certainly holds.
For the rest of the proof, assume that $m \geq 2$.

For now assume that $H \leq \GL(V)$ is primitive.
Let $\psi \colon \GL(V) \to \PGL(V)$ be the quotient map.
Let $N_0 = \ker\psi \cap N$ and note that $N_0$ is a normal subgroup of $H$ such that $N_0 \leq N$.
However, $N \in \T$ and $N/M$ is a chief factor of $H$, so either $N_0 \leq M$ or $N_0 = N$.
On the one hand, if $N_0 \leq M$, then $N/M$ is a quotient of $N/N_0 = N\psi$, which is a normal $q$-subgroup of the primitive group $H\psi \leq \PGL(V)$.
Now Lemma~\ref{lem:feit-tits} implies that $\order{N/M} \leq \order{N/N_0} \leq q^m \leq q^n$, so $s \leq m \leq n$.
Moreover, when $n \geq 5$, if $m \geq 5$ then
Lemma~\ref{lem:feit-tits} implies that $\order{N/M} \leq \order{N/N_0} < q^m \leq q^n$, while if $m < 5$, then $m < 5 \leq n$ immediately; in either case, $s < n$, as required.
On the other hand, if $N_0 = N$, then $N \leq \ker\psi$, so $N$, being a finite subgroup of $\ker\psi = \Centre{\GL(V)}$, is cyclic, which means that $N/M = C_q^s$ is cyclic, so $s = 1 < n$, as required.

It remains to assume that $H \leq \GL(V)$ is imprimitive.
In this case, let $\D$ be a direct sum decomposition $U_1 \oplus \cdots \oplus U_k$ of $V$ that is stabilized by $H$ and such that the natural representation $\phi_i \colon H_{U_i} \to \GL(U_i)$ is primitive (but not necessarily faithful) for all $i$. 
Note that $k \geq 2$ since $H$ is imprimitive.
Since $H$ is irreducible, $H$ transitively permutes the subspaces $U_1, \dots, U_k$, so, in particular, $k$ divides $m$ and $\dim U_i = m/k $ for all $i$.
Since $H$ stabilizes $\D$ we obtain the inclusion $H \leq \GL(V)_{\D} = \GL(U_1) \wr S_k$.
Let $\pi \colon H \to S_k$ be the corresponding projection map.
Let $K = \ker\pi \cap N$ and note that $K$ is a normal subgroup of $H$ such that $K \leq N$.
However, $N \in \T$ and $N/M$ is a chief factor of $H$, so either $K \leq M$ or $K = N$.

Suppose that $K \leq M$.  
In this case, $N/M$ is a quotient of $N/K$, which is isomorphic to the image of~$N$ in~$S_{k}$.  
Therefore $N/M$~is an abelian quotient of a permutation group of degree~$k$ and so, by the main theorem of~\cite{KovacsPraeger89}, $s \leq k/q$.
Hence $s < k \leq m \leq n$, as required.

It remains to consider the case when $K = N$.
Here, $N \leq \ker\pi$ and hence $N \leq H_{U_i}$ for all $i$.
For each $i$, let $L_i = \ker\phi_i \cap N$ and note that $L_i$ is a normal subgroup of $H$ such that $L_i \leq N$.
However, $N \in \T$ and $N/M$ is a chief factor of $H$, so for every~$i$ either $L_i \leq M$ or $L_i = N$.

Suppose that $L_i = N$ for all $i$.
Then $N \leq \ker\phi_i$ for all $i$, so $N \leq \ker\phi_1 \cap \cdots \cap \ker\phi_k$, but the representation $N \to \GL(V)$ is faithful and $V = U_1 \oplus \cdots \oplus U_k$, so the intersection $\ker\phi_1 \cap \cdots \cap \ker\phi_k$ is trivial.
This contradicts $N$ being nontrivial.

Therefore, we may fix some~$i$ such that $L_i \leq M$.
Write $X = N\phi_{i} \cong N/L_i$ and $Y = M\phi_{i} \cong M/L_i$, noting that $X/Y \cong N/M$.  
Then $X$~is a normal subgroup of $H_{U_{i}}\phi_{i} \leq \GL(U_{i})$.
Let $\psi_i \colon \GL(U_i) \to \PGL(U_i)$ be the quotient map.
Then $X\psi_i$ is a normal $q$-subgroup of $H_{U_i}\phi_i\psi_i$, which is a primitive subgroup of $\PGL(U_i)$.
Therefore, Lemma~\ref{lem:feit-tits} implies that $\order{X\psi_i} \leq q^{m/k}$.
Let $Z = \ker\psi_i \cap X$.
Then $\order{X/YZ} \leq \order{X/Z} = \order{X\psi_{i}} \leq q^{m/k}$.
Furthermore, $Z$~is cyclic (being a finite subgroup of $\Centre{\GL(U_{i})}$) and $YZ/Y$~is a subgroup of the elementary abelian $q$\nbd group~$X/Y$, so $\order{YZ/Y} = \order{Z/(Y \cap Z)} \leq q$.
Putting this together, $\order{N/M} = \order{X/Y} \leq q^{m/k + 1}$ and hence $s \leq m/k + 1 \leq m/2 + 1$ since $k \geq 2$.
Therefore, $s < n$ if $n \geq 3$, and in all cases $s \leq n$, as required.
\end{proof}

We are now in a position to prove the main theorem of this section.

\begin{proof}[Proof of Theorem~\ref{thm:width}:]
For ease of notation, write $N = N_{\ell - 1} = T^{n}$.
If $N$~is nonabelian, then $\Cent{G}{N} = 1$ as $N$~is the unique minimal normal subgroup.
If $N$~is abelian then, since $\Frat{G} = 1$, Lemma~\ref{lem:faithful} implies $\Cent{G/N}{N} = 1$.
In either case, we deduce that the conjugation action of $G$ on $N$ induces a homomorphism $\rho \colon G \to \Out{N}$ with $\ker\rho = N$.

\paragraph{Case~1} $T$~is nonabelian.

In this case, $\Out{N} = \Out{T} \wr S_n$.
Let $B = (\Out{T})^n$ be the base group of the wreath product $W = \Out{T} \wr S_n$.
Then $G\rho \cap B \normal G\rho$, and $G\rho \cap B$ is soluble because $\Out{T}$~is soluble.

\paragraph{Case~1a} $S$ is nonabelian.

Necessarily $j < k$, so $N_{j-1}/N_j = S^m$ is isomorphic to a section of $G\rho/(G\rho \cap B)$, which is a subgroup of $S_n$. 
Let $p$ be the greatest divisor of $\order{S}$.
Then $C_p^m$ is also a section of $S_n$, so the main result of \cite{KovacsPraeger89} forces $n \geq pm$ as claimed in part~\ref{it:width_nn}.

\paragraph{Case~1b} $S$ is abelian.

Write $S = C_p$ where $p$ is prime.
First assume that $G\rho \cap B \leq N_j$.
Then, as above, $N_{j-1}/N_j = C_p^m$ is isomorphic to section of $G\rho/(G\rho \cap B)$, which is a subgroup of $S_n$, so the main result of \cite{KovacsPraeger89} gives $n \geq pm$. 
Now assume that $G\rho \cap B > N_j$ and hence $G\rho \cap B \geq N_{j-1}$.
Let 
\[
\Out{T} = H_0 > H_1 > \cdots > H_s = 1
\]
be a chief series for $\Out{T}$.
This gives a normal series 
\[
W > B = H_0^n > H_1^n > \cdots > H_s^n = 1
\]
for $W$, which, in turn, gives a normal series
\begin{equation}
\label{eq:width_an}
G/N \cong G\rho = G\rho \cap W \geq G\rho \cap B = G\rho \cap H_0^n \geq G\rho \cap H_1^n \geq \cdots \geq G\rho \cap H_s^n = 1
\end{equation}
for $G/N$. 
Since $\T$ is a prescribed chief tail for $G$, the terms in the series
\[
G/N \geq N_0/N > N_1/N > \cdots > N_{\ell-1}/N = 1
\]
must occur in any chief series obtained by refining \eqref{eq:width_an}.
In particular, we can fix $i \in \{ 1, \dots, s \}$ such that $N_{j-1}/N_j = C_p^m$ is a section of $(G\rho \cap H_{i-1}^n)/(G\rho \cap H_i^n)$, which is isomorphic to a subgroup of $H_{i-1}^n/H_i^n \cong (H_{i-1}/H_i)^n$.
By Lemma~\ref{lem:rank_out}, either $H_{i-1}/H_i$ is cyclic or $T = {\rm P}\Omega_8^+(q)$ for some odd prime power $q$ and $H_{i-1}/H_i = C_2^2$. 
Therefore, either $m \leq n$ or $p = 2$, $T = {\rm P}\Omega_8^+(q)$ for some odd prime power $q$ and $m \leq 2n$. 
This completes the proof of part~\ref{it:width_an}.

\paragraph{Case~2} $T$ is abelian.

Write $T = C_p$ and note that $\Out{N} = \GL_n(p)$.
In particular, $G\rho$ is a subgroup of $\GL_n(p)$ isomorphic to $G/N$.
Moreover, $N$ is a minimal normal subgroup of $G$, so $G\rho$ is irreducible.
By Lemma~\ref{lem:frattini_chief_factors}, $N_{j}/N = \Frat{G/N}$ and if this is nontrivial then $N_{j}/N \leq \Op[q]{G/N} \leq N_{j-1}/N$ for some prime~$q$.
However, $G\rho$~is an irreducible subgroup of~$\GL_{n}(p)$ so, by Lemma~\ref{lem:no_normal_p-group}, it has no nontrivial normal $p$\nbd subgroup.  
Hence, if $N_{j}/N$~is nontrivial then it is a $q$\nbd group for some prime~$q \neq p$.

\paragraph{Case~2a} $S$ is nonabelian.

Then $N_{j-1}/N_{j} = S^{m}$ is perfect, so $N_{j-1}' = N_{j-1}$ as $\T$ is a prescribed chief tail.  
Since $\Cent{G/N_{\ell-1}}{N_{\ell-1}/N_{\ell}} = \Cent{G/N}{N} = 1$, certainly $N_{j-1}/N_{\ell-1}$~does not centralize~$N_{\ell-1}/N_{\ell}$; that is, $[N_{\ell-1},N_{j-1}] \nleq N_{\ell}$.  
Take $k \in \{ j+1,\dots,\ell \}$ to be minimal such that $[N_{k-1},N_{j-1}] \nleq N_{k}$.  
Thus $[N_{i-1},N_{j-1}] \leq N_{i}$ for $j+1 \leq i \leq k-1$.
In particular, $[N_{k-2},N_{j-1}] \leq N_{k-1}$.  
For all $i$ such that $j+1 \leq i \leq k-2$, if $[N_{i},N_{j-1}] \leq N_{k-1}$, then
\[
[N_{i-1},N_{j-1},N_{j-1}] \leq [N_{i},N_{j-1}] \leq N_{k-1},
\]
so, as $N_{j-1}$~is perfect, we deduce by the Three Subgroups Lemma that $[N_{i-1},N_{j-1}] \leq N_{k-1}$.  
Therefore, since $[N_{k-2},N_{j-1}] \leq N_{k-1}$, by induction, $[N_{j},N_{j-1}] \leq N_{k-1}$.  
Hence, $N_{j}/N_{k-1}$~is central in~$N_{j-1}/N_{k-1}$.

First assume that $k < \ell$.
Then the chief factor~$N_{k-1}/N_{k}$ occurs as a section of the normal $q$\nbd subgroup~$N_{j}/N$ of~$G/N$, which is isomorphic to an irreducible subgroup of~$\GL_{n}(p)$.
Furthermore, the images of $\T$ form a prescribed chief tail of $G/N$, and, therefore, $N_{k-1}/N_{k} = C_{q}^{s}$ for some positive integer~$s$ where $s \leq n$ by Proposition~\ref{prop:abelian_abelian}.
The conjugation action of $G/N_k$ on $N_{k-1}/N_k$ induces a representation $\lambda \colon G/N_k \to \GL_s(q)$.
Moreover, $N_{k-1}/N_k$ is a minimal normal subgroup of $G/N_k$, so $\im\lambda$ is irreducible.
In particular, $\Op[q]{\im\lambda} = 1$, but $N_j/N_k$ is a normal $q$-subgroup of $G/N_k$, so $N_j/N_k \leq \ker\lambda$.
Since $N_{j-1}/N_{k-1} \nleq \Cent{G/N_{k-1}}{N_{k-1}/N_k}$, we deduce that $\ker\lambda = N_j/N_k$. 
Therefore, the image of~$\lambda$ contains a subgroup isomorphic to $N_{j-1}/N_j \cong S^m$.
In particular, there exists a faithful projective representation $S^m \to \PGL_s(q)$.
Let $t$ be the smallest degree of a faithful projective representation of $S$ in positive characteristic.
Then \cite[Proposition~5.5.7(i)]{KleidmanLiebeck90} implies that $s \geq tm$.
Since $s \leq n$, we deduce that $n \geq s \geq tm$, as claimed in part~\ref{it:width_na}.

Now assume that $k = \ell$.
Then $N_{k-1}/N_k = N = C_p^n$, and $N_{j-1}\rho$ is a subgroup of $\GL_n(p)$ isomorphic to $N_{j-1}/N_{k-1} = N_{j-1}/N$, which is a perfect central extension of $N_{j-1}/N_j = S^m$.
Moreover, $N_{j-1}\rho$ is a normal subgroup of the irreducible group $G\rho$, so, by Clifford's Theorem, we can write $\rho|_{N_{j-1}} = \rho_1 \oplus \cdots \oplus \rho_u$ where $\rho_1, \dots, \rho_u$ are the homogeneous components of $\rho|_{N_{j-1}}$ whose respective irreducible components $\lambda_1, \dots, \lambda_u$ are pairwise nonisomorphic.
In addition, $G\rho$ acts transitively on the components $\rho_1, \dots, \rho_u$, so, in particular, $\lambda_1, \dots, \lambda_u$ all have the same dimension $d$, which necessarily divides $n/u$.
Furthermore, the transitive action of $G\rho$ implies that $\im\rho_1, \dots, \im\rho_u$ are isomorphic as abstract groups.
As an abstract group, $N_{j-1}\rho$ is a subdirect product of
$\im\rho_1 \times \cdots \times \im\rho_u$, so
\cite[Corollary~2.6]{HarperLiebeck25} implies that each $\im\rho_i$ is a perfect central extension of $S^r$ for some integer $r \geq m/u$.
Writing $H = \im\lambda_{1} \cong \im\rho_{1}$, we have a perfect central extension $H$ of $S^r$ (where $r \geq m/u$) such that $H \leq \GL_d(p)$ is irreducible.
In particular, there is a divisor $e$ of $d$ such that $H \leq \GL_e(p^{d/e})$ is absolutely irreducible (see \cite[Lemma~2.10.2]{KleidmanLiebeck90}), which yields a faithful absolutely irreducible representation $\lambda \colon H \to \GL_e(\FF_p)$ (where $e$ divides $d$ which divides $n/u$).
By Schur's Lemma, this yields a faithful absolutely irreducible projective representation $\overline{\lambda} \colon S^r \to \PGL_e(\FF_p)$.
Let $t$ be the smallest degree of a faithful projective representation of $S$ in positive characteristic.
Then \cite[Proposition~5.5.7(ii)]{KleidmanLiebeck90} implies that $e \geq t^r \geq tr$.
Since $e \leq n/u$ and $r \geq m/u$, we deduce that $n \geq eu \geq tru \geq tm$, which completes the proof of part~\ref{it:width_na}.

\paragraph{Case~2b} $S$ is abelian.

As we noted at the start of Case~2, if $N_{j}/N$~is nontrivial then there is a prime~$q \neq p$ such that $N_{j}/N \leq \Op[q]{G/N} \leq N_{j-1}/N$. 
Thus, there are three cases to consider:
\begin{enumerate}[{\rm (I)}]
\item 
$N_{j} = N$ and $N_{j-1}/N_{j} = S^{m} = C_{q}^{m}$ for some prime~$q$.  
Necessarily, $q \neq p$ by Lemma~\ref{lem:no_normal_p-group}.
\item 
$N_{j} \neq N$ and there is a prime~$q \neq p$ such that $\Op[q]{G/N} = N_{j-1}/N$.
\item 
$N_{j} \neq N$ and there is a prime~$q \neq p$ such that $\Op[q]{G/N} = N_{j}/N$.  
In this case $N_{j-1}/N_{j} = S^{m} = C_{r}^{m}$ for some prime~$r \neq q$.
\end{enumerate}

First consider cases~(I) and~(II).
In both cases, $N_{j-1}/N$~is a normal $q$\nbd subgroup of~$G/N$, which is an irreducible subgroup of~$\GL_{n}(p)$, and $N_{j-1}/N_{j} = C_{q}^{m}$ is a chief factor occurring in a prescribed chief tail. 
Proposition~\ref{prop:abelian_abelian} implies that $m < n$ if $n \geq 5$ and $m \leq n$ in all cases.

Therefore, it remains to consider case~(III).
Write $N_{j}/N_{j+1} = C_q^s$ for some $s \geq 1$.
Since $N \cong C_{p}^{n}$ is the minimal normal subgroup of~$G$, \ $G/N$~is isomorphic to an irreducible subgroup of~$\GL_{n}(p)$, so $s \leq n$ and if $n \geq 5$ then $s < n$.  

We claim that $\Cent{G/N_{j}}{N_{j}/N_{j+1}} = 1$.
To see this, suppose otherwise.
Then $N_{j-1}/N_j = C_r^m$ must centralize $N_j/N_{j+1} = C_q^s$.
Therefore, $Q = N_j/N_{j+1}$ is a central Sylow $q$-subgroup of $N_{j-1}/N_{j+1}$, which means that the Sylow $r$-subgroup $R$ of $N_{j-1}/N_{j+1}$ is normal in $N_{j-1}/N_{j+1}$, hence characteristic in $N_{j-1}/N_{j+1}$ and therefore normal in $G/N_{j+1}$. 
Hence, 
\begin{gather*}
G/N_{j+1} > N_{j-1}/N_{j+1} > Q > 1 \\
G/N_{j+1} > N_{j-1}/N_{j+1} > R > 1
\end{gather*}
are two normal series for $G/N_{j+1}$ with no common refinement, which contradicts $N_{j-1}$ being a member of the prescribed chief tail $\T$.
This proves the claim.

Since $\Cent{G/N_j}{N_j/N_{j+1}} = 1$, the conjugation action of $G/N_{j+1}$ on $N_j/N_{j+1}$ induces a representation $\lambda \colon G/N_{j+1} \to \GL_s(q)$ with $\im\lambda \cong G/N_j$.
Moreover, $N_j/N_{j+1}$ is a minimal normal subgroup of $G/N_{j+1}$, so $\im\lambda$ is irreducible.
Since $G/N_j$ is isomorphic to an irreducible subgroup of $\GL_s(q)$ and $N_{j-1}/N_j = C_r^m$ is a chief factor occurring in a prescribed chief tail, Proposition~\ref{prop:abelian_abelian} implies that $m \leq s$.
However, we know that $s < n$ if $n \geq 5$ and $s \leq n$ in all cases, and this establishes part~\ref{it:width_aa}.
\end{proof}

\subsection{Generation of groups with prescribed chief tail}
\label{ss:main}

We now turn to the proof of the main theorem of the paper.
Let us fix some notation for this section.
Let $G$ be a finite group with a prescribed chief tail $\T = (N_0, \dots, N_\ell)$.
Write
\[
\{ i_1, \dots, i_k \} = \set{i \in \{ 1, \dots, \ell \}}{\Frat{G/N_i} = 1},
\]
where $i_1 < \cdots < i_k$. 
(It might be that $k = 0$, which corresponds to $N_0 \leq \Frat{G}$, but this degenerate case causes no problems for the general arguments that follow.)
For $j \in \{ 1, \dots, k \}$, write
\[
N_{i_j-1}/N_{i_j} = T_j^{n_j},
\]
where $T_j$ is simple.
Let us record some useful information about the parameters $n_1, \dots, n_k$.

\begin{lemma}
\label{lem:width_cyclic}
Adopt the notation above.
Let $m \in \{0, \dots, k - 1\}$, let $j \in \{ 1, \dots, k - m \}$ and assume that $T_j, \dots, T_{j+m}$ are all cyclic.
Then $N_{i_j-1}/N_{i_{j+m}}$ is soluble of derived length at least $m + 1$.
\end{lemma}

\begin{proof}
By hypothesis, the chief factors $N_{i_j-1}/N_{i_j}$ are all abelian.
If $1 \leq i \leq \ell$ and $\Frat{G/N_i} \neq 1$, then $N_{i-1}/N_i \leq \Frat{G/N_i}$ since $N_{i-1}$ lies in the prescribed chief tail $\T$ and hence $N_{i-1}/N_i$ must also be abelian.
Therefore, all composition factors of $N_{i_j-1}/N_{i_{j+m}}$ are abelian, so $N_{i_j-1}/N_{i_{j+m}}$ is soluble.

We claim that for all $r \leq m$, the $r$th derived subgroup $N_{i_j-1}^{(r)}$ contains $N_{i_{j+r}-1}$.
To prove the claim, we proceed by induction, noting that the claim holds when $r = 0$.
Now assume that $N_{i_j-1}^{(r)} \geq N_{i_{j+r}-1}$ for some $r \in \{ 0, \dots, m - 1 \}$.
Then
\[
N_{i_j-1}^{(r)} \geq N_{i_{j+r}-1} > N_{i_{j+r}} \geq N_{i_{j+r+1}-1}
\]
so $N_{i_j-1}^{(r)}/N_{i_{j+r+1}-1}$ is a nontrivial normal subgroup of $G/N_{i_{j+r+1}-1}$. 
Since $N_{i_{j+r+1}-1} \in \T$, we may apply Lemma~\ref{lem:faithful}, to deduce that $G/N_{i_{j+r+1}-1}$ acts faithfully on $N_{i_{j+r+1}-1}/N_{i_{j+r+1}}$.
Hence, the nontrivial subgroup $N_{i_j-1}^{(r)}/N_{i_{j+r+1}-1}$ does not centralize $N_{i_{j+r+1}-1}/N_{i_{j+r+1}}$.
Therefore,
\[
N_{i_j-1}^{(r+1)} \geq [ N_{i_j-1}^{(r)}, N_{i_{j+r+1}-1} ] = N_{i_{j+r+1}-1}.
\]
This completes the induction and proves the claim.

In particular,
\[
N_{i_j-1}^{(m)} \geq N_{i_{j+m}-1} > N_{i_{j+m}},
\]
so the derived length of $N_{i_j-1}/N_{i_{j+m}}$ is at least $m + 1$, as required.
\end{proof}

\begin{proposition}
\label{prop:width_small}
Adopt the notation above. 
There exists $m \leq 15$ such that $n_m \geq 5$.
\end{proposition}

\begin{proof}
Let $m \in \{ 1, \dots, k \}$ and assume that $n_1, \dots, n_m$ are all at most $4$.
We will prove that $m \leq 14$, which establishes the result we desire.

We claim that at most one of $T_1, \dots, T_m$ is nonabelian.
To prove the claim, we suppose for a contradiction that there exist integers $1 \leq r < s \leq m$ such that $T_r$ and $T_s$ are both nonabelian.
Now $G/N_{i_s}$ acts faithfully by conjugation on $N_{i_s-1}/N_{i_s} = T_s^{n_s}$ and hence $G/N_{i_s-1}$ embeds in $\Out{T_s} \wr S_{n_s}$.
Since $T_r$ is nonabelian, $G/N_{i_s-1}$ is insoluble, but $\Out{T_s}$ is soluble, so $S_{n_s}$ must be insoluble and hence $n_s \geq 5$, which is contrary to our hypothesis.
This proves the claim.

For now assume that $T_1, \dots, T_m$ are all abelian.
Then Lemma~\ref{lem:width_cyclic} implies that $N_{i_1-1}/N_{i_{m-1}}$ is soluble of derived length at least $m - 1$.
Consequently, $N_{i_1-1}/N_{i_m-1}$ is also soluble of derived length at least $m - 1$.
Now $G/N_{i_m-1}$ embeds in $\Aut{(T_m^{n_m})}$ and $\Aut{(T_m^{n_m})} = \GL_{n_m}(p)$ for some prime $p$, so a result of Huppert~\cite[Satz~9]{Huppert57} implies that $m - 1 \leq 2 n_m \leq 8$, so $m \leq 9$.

It remains to assume that there exists a unique integer $r \in \{ 1, \dots, m \}$ such that $T_r$ is nonabelian.
Then $G/N_{i_r}$ acts faithfully by conjugation on $N_{i_r-1}/N_{i_r} = T_r^{n_r}$, so $G/N_{i_r-1}$ embeds in $\Out{T_r} \wr S_{n_r}$.
It follows quickly from the description of outer automorphism groups of simple groups in \cite[Theorem~2.5.12]{GorensteinLyonsSolomon98}, for example, that $\Out{T_r}$ has derived length at most $3$.

For now assume that $r = m$.
Then $n_m = n_r \leq 4$ and the derived length of $G/N_{i_m-1} = G/N_{i_r-1}$ is at most the derived length of $\Out{T_r} \wr S_4$, which is at most $6$.
The derived length of $G/N_{i_m-1}$ is at least that of $G/N_{i_{m-1}}$, which, by Lemma~\ref{lem:width_cyclic}, is at least $m - 1$.
Therefore, $m - 1 \leq 6$, so $m \leq 7$.

It remains to assume that $r < m$.
By Theorem~\ref{thm:width}, $n_{r+1} \geq 2n_r$. 
In particular, $n_r \leq 2$ since $n_{r+1} \leq 4$.
Therefore, the derived length of $G/N_{i_r-1}$ is at most the derived length of $\Out{T_r} \wr S_2$, which is at most $4$.
The derived length of $G/N_{i_r-1}$ is at least that of $G/N_{i_{r-1}}$, which, by Lemma~\ref{lem:width_cyclic}, is at least $r - 1$.
Therefore, $r - 1 \leq 4$.
Since $T_{r+1}, \dots, T_{m-1}$ are all abelian, Lemma~\ref{lem:width_cyclic} implies that $N_{i_{r+1}-1}/N_{i_{m-1}}$ is a soluble group of derived length at least $m - r - 1$.
However, $N_{i_{r+1}-1}/N_{i_{m-1}}$ embeds in $\GL_{n_m}(p)$ for some prime $p$, so \cite[Satz~9]{Huppert57} implies that $m - r - 1 \leq 2n_m \leq 8$.
Therefore, $m \leq 14$.
This completes the proof.
\end{proof}

\begin{proposition}
\label{prop:width_grow}
Adopt the notation above.
Let $j \in \{ 1, \dots, k - 1 \}$ and assume that $n_j \geq 5$.
Then one of the following holds
\begin{enumerate}
\item \label{it:width_grow_increase}
$n_{j+1} \geq n_j + 1$.
\item \label{it:width_grow_decrease}
$n_j \geq n_{j+1} \geq \half n_j$ and, if $j \leq k - 2$, then $n_{j+2} \geq n_j + 2$.
\end{enumerate}
\end{proposition}

\begin{proof}
Let $j \in \{ 1, \dots, k - 1 \}$ such that $n_j \geq 5$.
Assume that part~\ref{it:width_grow_increase} does not hold, so $n_{j+1} \leq n_j$.
Then, by Theorem~\ref{thm:width}, $T_j$ is abelian and $T_{j+1}$ is nonabelian.
Moreover, $n_{j+1} \geq \frac{1}{2}n_j$ and unless $T_{j+1} = \mathrm{P}\Omega_8^+(q)$ for an odd prime power $q$, then in fact $n_{j+1} = n_{j}$.
It remains to assume that $j \leq k - 2$.
First assume that $T_{j+2}$ is nonabelian.
Then since $T_{j+1}$ is nonabelian, Theorem~\ref{thm:width}\ref{it:width_nn} implies that $n_{j+2} \geq 5n_{j+1} \geq \frac{5}{2}n_j \geq n_j + 2$.
Now assume that $T_{j+2}$ is abelian.
Theorem~\ref{thm:width}\ref{it:width_na} implies that $n_{j+2} \geq tn_{j+1}$ where $t$~is the smallest degree of a faithful projective representation of~$T_{j+1}$ in positive characteristic.
If $T_{j+1} \neq \mathrm{P}\Omega_{8}^{+}(q)$ for an odd prime power $q$, then $n_{j+1} = n_{j}$, so we deduce $n_{j+2} \geq 2n_{j+1} \geq n_{j}+2$.
On the other hand, if $T_{j+1} = \mathrm{P}\Omega_8^+(q)$ for an odd prime power $q$, then it is easy to see that $t \geq 3$ (see \cite[Theorem~5.3.9 \& Proposition~5.4.13]{KleidmanLiebeck90}, for example), so $n_{j+2} \geq 3n_{j+1} \geq \frac{3}{2}n_j \geq n_j + 2$.
\end{proof}

\begin{corollary}
\label{cor:width}
Adopt the notation above.
There exists $m \in \{ 1, \dots, 15 \}$ such that $n_{m} \geq 5$ and, for all $i \in \{ 1, \dots, k - m \}$, the following hold
\begin{enumerate}
\item \label{it:width_increase}
If $n_{m+i} > n_{m+i-1}$, then $n_{m+i} \geq n_m + i$.
\item \label{it:width_decrease}
If $n_{m+i} \leq n_{m+i-1}$, then $n_{m+i} \geq \frac{1}{2}(n_m + i - 1)$.
\item \label{it:width}
In all cases, $n_{m+i} \geq \frac{1}{3}(n_m + i)$.
\end{enumerate}
\end{corollary}

\begin{proof}
By Proposition~\ref{prop:width_small}, we can fix $m \in \{ 1, \dots, 15 \}$ such that $n_m \geq 5$.
We will prove parts~\ref{it:width_increase} and~\ref{it:width_decrease} together by induction on $i$.
If $i = 1$, then both parts hold by Proposition~\ref{prop:width_grow}.
Now assume that $i \geq 2$ and that both parts hold for smaller values of $i$.
First assume that $n_{m+i-1} > n_{m+i-2}$. 
Then, by induction, $n_{m+i-1} \geq n_m + i - 1$, so if $n_{m+i} > n_{m+i-1}$, then $n_{m+i} \geq n_m + i$, and if $n_{m+i} \leq n_{m+i-1}$, then $n_{m+i} \geq \frac{1}{2}(n_m + i - 1)$, as required.
Now assume that $n_{m+i-1} \leq n_{m+i-2}$.
By Proposition~\ref{prop:width_grow}, this means that $n_{m+i} \geq n_{m+i-2} + 2$.
In particular, the claim holds if $i = 2$.  
If $i > 2$ then, as it is also clear from Proposition~\ref{prop:width_grow} that the sequence $(n_m, \dots, n_k)$ does not weakly decrease twice in a row, necessarily $n_{m+i-2} > n_{m+i-3}$, which, by induction, means that $n_{m+i-2} \geq n_m + i - 2$.
Therefore, $n_{m+i} \geq n_{m+i-2} + 2 \geq n_m + i - 2 + 2 = n_m + i$, as required.
This completes the proof of parts~\ref{it:width_increase} and~\ref{it:width_decrease}, and combining these gives part~\ref{it:width} noting that $n_m \geq 5$.
\end{proof}

Before proving our main result, we need one final technical ingredient.
For a characteristically simple group $N$, write
\begin{equation}
\label{eq:aut_tilde}
\wAut(N) = \left\{ 
\begin{array}{ll}
N \semidirect \Aut{N} & \text{if $N$ is abelian} \\
\Aut{N}               & \text{if $N$ is nonabelian.}
\end{array}
\right.
\end{equation}
The significance of this definition is that if $\Frat{G/N_i} = 1$, then it follows that $G/N_{i}$~embeds in $\wAut(N_{i-1}/N_{i})$.  
When $N_{i-1}/N_{i}$~is nonabelian this holds because $G/N_{i}$~acts faithfully on its minimal normal subgroup.  
When $N_{i-1}/N_{i}$~is abelian, we use Lemmas~\ref{lem:complemented} and~\ref{lem:faithful} to tell us there is a complement for~$N_{i-1}/N_{i}$ that acts faithfully on the minimal normal subgroup and hence establish that $G/N_{i}$~embeds in the semidirect product~$\wAut(N_{i_1}/N_{i})$.
The following lemma will be useful.

\begin{lemma}
\label{lem:alpha_growth}
For every~$A > 0$, there exists~$B > 0$ such that if $T$~is a finite simple group and $n$~is a positive integer with $\order{\wAut(T^{n})} \geq B$, then $\alpha(T)^{n} \geq A$.
\end{lemma}

\begin{proof}
Let $A > 0$.
Let $n \geq 1$.

If $T = C_p$, then $\alpha(T)^n \to \infty$ as $n \to \infty$ or $p \to \infty$, so there exist $N_1, P > 0$ such that $\alpha(T)^n \geq A$ if $n \geq N_1$ or $p \geq P$.
Let $B_{1}$ be the maximum of~$\order{\wAut(C_{p}^{n})}$ where $n \leq N_1$ and $p \leq P$.

If $T = A_m$, then $\alpha(T)^n \to \infty$ as $n \to \infty$ or $m \to \infty$, so there exist $N_2, M > 0$ such that $\alpha(T)^n \geq A$ if $n \geq N_2$ or $m \geq M$.
Let $B_{2}$ be the maximum of~$\order{\wAut(A_{m}^{n})}$ where $n \leq N_2$ and $m \leq M$.

If $T$ is a group of Lie type of rank $r$ over $\F_q$, then $\alpha(T)^n \to \infty$ as $n \to \infty$ or $r \to \infty$ or $q \to \infty$, so there exist $N_3, Q, R > 0$ such that $\alpha(T)^n \geq A$ if $n \geq N_3$ or $q \geq Q$ or $r \geq R$.
Let $B_{3}$ be the maximum of~$\order{\wAut(T^{n})}$ where $n \leq N_3$ and $T$ is of Lie type of rank $r \leq R$ over $\F_q$ with $q \leq Q$.

If $T$ is a sporadic group, then $\alpha(T)^n \to \infty$ as $n \to \infty$, so there exists $N_4 > 0$ such that $\alpha(T)^n \geq A$ if $n \geq N_4$.
Let $B_{4}$ be the maximum of~$\order{\wAut(T^{n})}$ where $n \leq N_4$ and $T$ is sporadic.

Let $B = \max\{ B_{1}, B_{2}, B_{3}, B_{4} \} + 1$.
By construction, if $T$ is a finite simple group and $n$ is a positive integer and $\order{\wAut(T^{n})} \geq B$, then $\alpha(T)^n \geq A$, as required.
\end{proof}

We are now in a position to prove the main result of this section.

\begin{proof}[Proof of Theorem~\ref{thm:main}:]
Let $\epsilon > 0$.
By Lemma~\ref{lem:alpha_bound}, we can fix $C > 1$ such that $\alpha(T) \geq C$ for all finite simple groups~$T$.
Choose an integer~$r$ such that $C^{-(5+r)/6} < (1-C^{-1/6})\epsilon/2$ and then define $A = ( 2(15+r)/\epsilon)^{2}$.
By Lemma~\ref{lem:alpha_growth}, we can fix $c > 0$ such that, for all finite simple groups~$T$ and positive integers~$n$, if $\order{\wAut{(T^n)}} \geq c$, then $\alpha(T)^{n} \geq A$.        

Now let $d \geq 2$ and let $G$~be a finite group with a prescribed tail $\T$.
Let $N \in \T$ such that $\order{G/N} \geq c$.
If $P_d(G) = 0$, then $d < d(G)$, which, by Corollary~\ref{cor:lucchini_menegazzo} means that $d < d(G/N)$, so $P_d(G/N) = 0$, and the theorem holds.
Therefore, we may assume that $d \geq d(G)$.
As above, write $\T = (N_0, \dots, N_\ell)$ and 
\[
\{ i_{1}, \dots, i_{k} \} = \set{ i \in \{ 1, \dots, \ell \}}{\Frat{G/N_{i}} = 1},
\]
where $i_{1} < \cdots < i_{k}$, and for $j \in \{ 1, \dots, k \}$, write
\[
N_{i_{j}-1}/N_{i_{j}} = T_{j}^{n_{j}},
\]
where $T_{j}$~is simple.  
By Corollary~\ref{cor:width}, we can fix $m \leq 15$ such that $n_{m} \geq 5$ and $n_{m+j} \geq \frac{1}{3}(n_{m}+j)$ for all $j \in \{1, \dots, k-m\}$.

Fix the index $t \in \{ 0, \dots, \ell \}$ such that $N = N_{t}$.  
Then
\[
\frac{P_d(G)}{P_d(G/N)} = \prod_{i = t+1}^{\ell} \frac{P_d(G/N_{i})}{P_d(G/N_{i-1})}.
\]
If $N \leq \Frat{G}$, then Corollary~\ref{cor:prob_frattini} implies that $P_d(G) = P_d(G/N)$, and hence the theorem holds.
Therefore, we may assume that $N \not\leq \Frat{G}$.
This means that $\Frat{G/N_i} = 1$ for some $t+1 \leq i \leq \ell$, since, otherwise, $N_{i-1}/N_i \leq \Frat{G/N_i}$ for all $t+1 \leq i \leq \ell$, which inductively implies that $N = N_t \leq \Frat{G}$.
In particular, this ensures that $k > 0$ and there exists $s \in \{ 1, \dots, k \}$ such that $i_s \geq t+1$; fix $s$ minimal subject to this condition.
If $t + 1 \leq i \leq \ell$ and $\Frat{G/N_{i}} \neq 1$, then $N_{i-1}/N_{i} \leq \Frat{G/N_{i}}$ so, by Corollary~\ref{cor:prob_frattini}, $P_d(G/N_{i-1}) = P_d(G/N_{i})$. 
Therefore
\[
\frac{P_d(G)}{P_d(G/N)} = \prod_{j = s}^{k} \frac{P_d(G/N_{i_{j}})}{P_d(G/N_{i_{j}-1})}.
\]
By Lemma~\ref{lem:prob} and Theorem~\ref{thm:zeta},
\begin{align}
\frac{P_d(G)}{P_d(G/N)} \geq \prod_{j=s}^{k} \bigl( 1 - \zeta_{G/N_{i_{j}}, N_{i_{j}-1}/N_{i_{j}}}(d) \bigr) 
\geq \prod_{j=s}^{k} \bigl( 1 - \alpha(T_{j})^{-n_{j}/2} \bigr)
\geq 1 - \sum_{j=s}^{k} \alpha(T_{j})^{-n_{j}/2}, \label{eq:ProbEstimate}
\end{align}
where $\alpha$~is the function defined in Equation~\eqref{eq:f}.  

We now establish two claims about the terms in the above sum.

\paragraph{Claim~1} $\sum_{j=m+r}^{k} \alpha(T_{j})^{-n_{j}/2} < \epsilon/2$.

\medskip

We interpret the sum as~$0$ if $m+r > k$, in which case the inequality holds immediately.  
Otherwise, note that $n_{m+i} \geq \frac{1}{3}(5+i)$ for all $i \geq 1$, so, as $\alpha(T_{j}) \geq C$ for all~$j$, we compute that
\[
\sum_{\mathclap{m+r \leq j \leq k}} \; \alpha(T_{j})^{-n_{j}/2} \leq
\sum_{i=r}^{\infty} C^{-(5+i)/6} = 
\frac{C^{-(5+r)/6}}{1 - C^{-1/6}} < \frac{\epsilon}{2}
\]
by our choice of~$r$.

\paragraph{Claim~2} $\sum_{j=s}^{m+r-1} \alpha(T_{j})^{-n_{j}/2} < \epsilon/2$.

\medskip

As in the previous claim, if $m+r-1 < s$, then this sum is interpreted as~$0$.
Otherwise, let $j$~be an integer with $s \leq j < m+r$.  
Then $\Frat{G/N_{i_{j}}} = 1$ so, as we noted earlier, $G/N_{i_{j}}$~embeds in~$\wAut(N_{i_{j}-1}/N_{i_{j}})$.  
In particular,
\[
c \leq \order{G/N} = \order{G/N_{t}} \leq 
\order{G/N_{i_{s}}} \leq \order{G/N_{i_{j}}} \leq
\order{\wAut(N_{i_{j}-1}/N_{i_{j}})} = \order{\wAut(T_{j}^{n_{j}})}.
\]
Hence, our choice of~$c$ guarantees that $\alpha(T_{j})^{n_{j}} \geq A$.
Therefore,
\[
\sum_{j=s}^{m+r-1} \alpha(T_{j})^{-n_{j}/2} \leq (m+r-s) A^{-1/2} < (15+r) A^{-1/2} = \frac{\epsilon}{2}.
\]

We have established both claims, so, upon substituting into Equation~\eqref{eq:ProbEstimate}, we deduce
\[
P_d(G) \geq (1 - \epsilon) \, P_d(G/N),
\]
which completes the proof.
\end{proof}

It just remains to prove the consequences of our main result.
We focus on finite groups in this section and profinite groups in the next.
First observe that Theorem~\ref{thm:uniserial} is an immediate consequence of Theorem~\ref{thm:main}.
We now turn to Corollary~\ref{cor:uniserial} and Remark~\ref{rem:uniserial}.

\begin{corollary}
\label{cor:limit}
Let $d \geq 2$.
Let $(G_n)_n$ be a sequence of finite groups.
For each $n$, let $\T_n$ be a prescribed chief tail of $G_n$, let $N_n \in \T_n$ and let $Q_n = G_n/N_n$.
Assume that $\order{Q_n} \to \infty$ and $P_d(Q_n) \to a$ as $n \to \infty$.
Then $P_d(G_n) \to a$ as $n \to \infty$.
\end{corollary}

\begin{proof}
Let $\epsilon > 0$.
By Theorem~\ref{thm:main}, there exists $c > 0$ such that if $H$~is a finite group with a prescribed chief tail $\T$ and $N \in \T$ such that $\order{H/N} \geq c$, then $P_d(H) \geq (1 - \epsilon) \, P_d(H/N)$.  
Since $\order{Q_n} \to \infty$ as $n \to \infty$, we can fix $m$ such that $\order{Q_n} \geq c$ for all $n \geq m$.
Therefore, $(1 - \epsilon) \, P_d(Q_n) \leq P_d(G_n) \leq P_d(Q_n)$ for all $n \geq m$, which proves that $P_d(G_n)/P_d(Q_n) \to 1$ as $n \to \infty$. 
By hypothesis, $P_d(Q_n) \to a$ as $n \to \infty$, so we conclude that $P_d(G_n) \to a$ as $n \to \infty$.
\end{proof}

Corollary~\ref{cor:limit} when combined with the main theorem of \cite{LiebeckShalev95} on simple groups gives Corollary~\ref{cor:uniserial}, and when combined with the main theorem of \cite{Dixon69} on symmetric groups it gives the claim in Remark~\ref{rem:uniserial}.

\subsection{Generation of profinite groups}
\label{ss:profinite}

We now prove our results on profinite groups, beginning with the following technical tool.

\begin{lemma}
\label{lem:profinite}
Let $G$ be a profinite group with finitely many chief series.  
Then $G$ has an open normal subgroup~$N$ such that the following hold
\begin{enumerate}
\item 
every open normal subgroup of~$G$ is comparable with~$N$
\item 
every pair of open normal subgroups of~$G$ contained in~$N$ are comparable with each other.
\end{enumerate}
\end{lemma}

\begin{proof}
For a contradiction, suppose that there is no open normal subgroup satisfying (i) and~(ii) in the statement.  
Let $K$ be an open normal subgroup of $G$ and let $n$ be the number of distinct chief series of the finite group $G/K$.

First suppose that there is an open normal subgroup~$L$ of~$G$ that is incomparable with~$K$.  
For convenience, write $\overline{G} = G/(K \cap L)$, and for $X \leq G$, write $\overline{X}$ for the image of $X$ in $\overline{G}$.
Then $\overline{G}$ has at least $n$~distinct chief series containing~$\overline{K}$ as a term, and none of these contain~$\overline{L}$ since $L$ is incomparable with $K$.  
Furthermore, by refining the open normal series $\overline{G} > \overline{L} > 1$, there is at least one chief series with $\overline{L}$ as a term.  
Therefore, $\overline{G}$ has at least $n+1$~distinct chief series.

Now suppose that every open normal subgroup of~$G$ is comparable with~$K$.  
Then $K$ satisfies (i) in the statement, so, by supposition, $K$ does not satisfy (ii).
Said otherwise, there exist incomparable open normal subgroups $L_{1}$ and~$L_{2}$ contained in~$K$.  
For convenience, write $\overline{G} = G/(L_{1} \cap L_{2})$, and for $X \leq G$, write $\overline{X}$ for the image of $X$ in $\overline{G}$.
For $i \in \{1, 2\}$, there are at least $n$~distinct chief series of~$\overline{G}$ arising by refining the open normal series $\overline{G} > \overline{K} > \overline{L_{i}} > 1$.
Since $L_{1}$ and~$L_{2}$ are incomparable, we conclude that $\overline{G}$ has at least $2n$~distinct chief series.

Consequently, by repeated application of this argument, there exist finite quotients of~$G$ with a arbitrarily many distinct chief series, which contradicts the hypothesis that $G$~has finitely many chief series.
\end{proof}

\begin{proof}[Proof of Theorem~\ref{thm:profinite}:]
Let $G$ be a profinite group with finitely many chief series.
By Lemma~\ref{lem:profinite}, we may fix an open normal subgroup~$N$ of~$G$ satisfying the two conditions of that lemma.
Let $G = G_0 > G_1 > \cdots$ be a chief series of $G$ refining the open normal series $G > N > 1$ and fix~$r$ such that $G_{r} = N$.

We first prove that $d(G) \leq \max\{ 2, d(G/N) \}$.
Let $K$ be an open normal subgroup of~$G$.
If $N \leq K$, then immediately $d(G/K) \leq d(G/N)$.
Otherwise, $K \leq N$ and $K = G_{m}$ for some $m \geq r$.  
The properties of~$N$ ensure that $G_{r}/K = N/K$ is the first term of a prescribed chief tail $(G_{r}/K, \dots, G_{m}/K)$ for~$G/K$, so Corollary~\ref{cor:lucchini_menegazzo} implies that $d(G/K) = \max \{ 2, d(G/N) \}$.
Therefore, in all cases, $d(G/K) \leq \max\{2, d(G/N)\}$.
Since $d(G) = \sup\{ d(G/K) \mid \text{$K \normal_{\rm o} G$} \}$, we conclude that $d(G) \leq \max\{2, d(G/N)\}$, as required.

Let $d \geq d(G)$.
We will now prove that $P_d(G) > 0$.
First assume that $d(G) = 1$.
Then $G$~is procyclic, so it is the Cartesian product of its Sylow pro-$p$ subgroups, each of which is either finite cyclic or isomorphic to the additive group~$\padics$ of $p$\nbd adic integers.  
If the Sylow pro-$p$ subgroup of~$G$ is nontrivial for infinitely many
primes~$p$, then $G$~has infinitely many normal series of the form $G >
G^{p} > 1$ and hence infinitely many chief series.  
If there are distinct primes $p$~and~$q$ such that the Sylow pro\nbd$p$ subgroup is
isomorphic to~$\padics$ and the Sylow pro\nbd$q$ subgroup is
nontrivial, then we can construct infinitely many chief series that
begin in the following way
\[
G > G^{p} > G^{p^{2}} > \dots > G^{p^{r}} > G^{p^{r}q} > \cdots.
\]
Hence, in order to have finitely many chief series, $G \cong \padics$
for some prime~$p$.  
Then any element not in the unique maximal normal subgroup of~$G$ is a topological generator and therefore $P_{d}(G) = \bigl(1 - \tfrac{1}{p}\bigr)^{d} > 0$.

It remains to assume that $d(G) \geq 2$.
By Theorem~\ref{thm:main}, there exists $c > 0$ such that if $H$~is a finite group with a prescribed chief tail~$\T$ and $L \in \T$ such that $\order{H/L} \geq c$, then $P_d(H) \geq \half P_d(H/L)$.  
Fix~$m$ such that $G_m \leq N$ and $\order{G/G_m} \geq c$.
Let $K$ be an open normal subgroup of~$G$.
If $G_m \leq K$, then immediately $P_d(G/K) \geq P_d(G/G_m)$.
Otherwise, since every open normal subgroup of $G$ is comparable with $N$ and every pair of open normal subgroups of $G$ contained in $N$ are comparable with each other, $K = G_n$ for some $n \geq m$.
In particular, $G_n \leq N$, so $(N/G_n, \dots, G_n/G_n)$ is a prescribed chief tail of~$G/G_n$.
By construction, $G_m/G_n$ is a term of this prescribed chief tail, $G/G_m = (G/G_n)/(G_m/G_n)$ and $\order{G/G_m} \geq c$.
Hence, Theorem~\ref{thm:main} gives $P_d(G/G_n) \geq \half P_d(G/G_m)$.
Therefore, in all cases, $P_d(G/G_n) \geq \half P_d(G/G_m)$.
Since $P_d(G) = \inf\{ P_d(G/K) \mid \text{$K \normal_{\rm o} G$} \}$ (see \cite[Theorem~1]{Mann96}), we conclude that $P_d(G) \geq \half P_d(G/G_m) > 0$, noting that $G/G_m$ is a $d$-generated finite group.
\end{proof}

\subsection{Examples and constructions of groups with prescribed chief tail}
\label{ss:examples}

We finish by presenting examples and constructions of finite groups with prescribed chief tail.  
These, in particular, justify the claims in Example~\ref{ex:main}.
Indeed, if one applies them to a uniserial group (for example, a finite simple group) then the result will be a uniserial group.
We begin with affine groups.

\begin{lemma}
\label{lem:affine_construction}
Let $G$~be a finite group and let $V$~be uniserial $\F_{p}G$\nbd module for some prime~$p$ with composition series $V = V_{0} > V_{1} > \dots > V_{k} = 0$.
Then $(V_{0},\dots,V_{k})$~is a prescribed chief tail of the affine group~$V \semidirect G$ if and only if $G$~acts faithfully on $V/V_{1}$.  
Moreover, if $G$~acts faithfully on $V/V_{1}$ and $(N_{0},\dots,N_{\ell})$~is a prescribed chief tail of~$G$, then $(V \semidirect N_{0}, \dots, V \semidirect N_{\ell}, V_{1}, \dots, V_{k})$ is a prescribed chief tail of $V \semidirect G$.
\end{lemma}

\begin{proof}
Certainly $V_{1} > V_{2} > \dots > V_{k}$ is a chain of normal subgroups of the semidirect product~$H = V \semidirect G$.
Suppose first that $G$~acts faithfully on~$V/V_{1}$.  
Let $N$~be any normal subgroup of the semidirect product~$H$.  
If $N \leq V$, then $N$~is a submodule of~$V$, so $N = V_{j}$ for some~$j$.  
Assume then that $N \nleq V$.  
Let $g$~be the image in~$G$ of some element $x \in N \setminus V$ and observe that the commutator of~$x$ with $v \in V$ is equal to $[v,x] = v^{g}-v$.
Since $g$~induces a nontrivial map on~$V/V_{1}$, we deduce that $[V,N] \nleq V_{1}$.  
Hence, the normal subgroup~$V \cap N$ is not contained in~$V_{1}$ and we deduce $V \leq N$.  
This establishes that $(V_{0},\dots,V_{k})$ is a prescribed chief tail for~$H$.

Suppose instead that $G$~does not act faithfully on~$V/V_{1}$.  
Let $K$~be the (nontrivial) kernel of the action on this quotient. 
A straightforward calculation then verifies that $V_{1} \semidirect K$ is a normal subgroup of $H$.  
Refining $H > V_{1} \semidirect K > 1$ to a chief series for~$G$ then produces a series that does not contain $V = V_{0}$ as one of its terms.  
Hence, $(V_{0},\dots,V_{k})$~is not a prescribed chief tail for~$H$.

The last assertion in the statement follows immediately by the Correspondence Theorem.
\end{proof}

\begin{example}
\label{ex:affine_symmetric}
Let $V = \F_{p}^n$ be the permutation module for the natural action of the symmetric group $S_n$ on $n \geq 5$ points.  
A straightforward calculation shows that the only proper nonzero submodules of $V$~are
\[
V_{1} = \biggl\{ \, (x_{1},x_{2},\dots,x_{n}) \in V \;\biggm|\; \sum_{i=1}^{n} x_{i} = 0 \, \biggr\}
\AND
V_{2} = \set{ (x,x,\dots,x) }{ x \in \F_{p} }.
\]

Suppose first that $p$ does not divide $n$.
Then $V_{1} \cap V_{2} = 0$, so $V = V_{1} \oplus V_{2}$, which means that $V$ is not a uniserial module and $V \semidirect S_{n}$ is not a uniserial group.  
Indeed its only prescribed chief tail is the trivial one~$(1)$.

Suppose instead that $p$ divides $n$.
Then $V_{2} \leq V_{1}$, so $V$ is a uniserial module with composition series $V > V_{1} > V_{2} > 0$.
In this case, $S_{n}$~does not act faithfully on~$V/V_{1}$, so $V \semidirect S_{n}$ is not uniserial.  
However, $S_{n}$~does act faithfully on~$V_{1}/V_{2}$, by our assumption on~$n$, and hence $V_{1} \semidirect S_{n}$ is uniserial.
\end{example}

\begin{example}
\label{ex:affine_equality}
Let $p$ be a prime satisfying $p \equiv \pm 1 \pmod{8}$.
Viewing $\F_p^4$ as a tensor product $\F_p^2 \otimes \F_p^2$ yields the embedding $(\GL_2(p) \circ \GL_2(p)) \semidirect S_2 \leq \GL_4(p)$. 
The congruence condition on~$p$ ensures that $\GL_2(p)$ has a subgroup in the Aschbacher class $\C_6$ that is a nonsplit extension $2^{1+2}_{-}.S_3$, which can also be viewed as a central extension $2.S_4$ (see \cite[Table~8.1]{BrayHoltRoneyDougal13}).
This gives a subgroup $2.(S_4 \wr S_2)$ of $(\GL_2(p) \circ \GL_2(p)) \semidirect S_2$ and, hence, of $\GL_4(p)$. 
Since $2.A_4$ is an irreducible subgroup of $\GL_2(p)$, the index two subgroup $H = 2.(A_4^2 \semidirect C_4)$ is an irreducible subgroup of $\GL_4(p)$.
One can easily check directly that $H$ is uniserial with chief factors
\[
C_2 \quad (C_2) \quad C_3^2 \quad C_2^4 \quad (C_2),
\]
where the Frattini chief factors in brackets.
Therefore, by Lemma~\ref{lem:affine_construction}, the affine group $G = p^4 \semidirect H$ is uniserial with chief factors
\[
C_2 \quad (C_2) \quad C_3^2 \quad C_2^4 \quad (C_2) \quad C_p^4,
\]
where, again, the Frattini chief factors in brackets.
\end{example}

We now turn to wreath products, with the following lemma presenting the simplest case.

\begin{lemma}
\label{lem:wreath_construction}
Let $G$~be a uniserial finite group, and let $G$ act faithfully and transitively on a finite set~$\Omega$.
Let $T$~be a nonabelian finite simple group.
Then the base group~$N$ of the wreath product $W = T \wr_\Omega G$ is its unique minimal normal subgroup of~$W$.  
If $(N_{0},\dots,N_{\ell})$ is a prescribed chief tail of~$G$, then $(N \semidirect N_{0}, \dots, N \semidirect N_{\ell}, 1)$ is a prescribed chief tail of~$W$.  
In particular, if $G$~is uniserial then $W$~is uniserial.
\end{lemma}

\begin{proof}
Write $W = T \wr_\Omega G$ and let $N = T^n$ be the base group $W$, where $n = \order{\Omega}$.
Since $G$ acts faithfully and transitively on $\Omega$, the conjugation action of $W$ on the set of $n$ simple direct factors of $N$ is transitive and has kernel $N$. 
Since $W$ acts transitively on the simple direct factors of $N$ and $T$ is a nonabelian simple group, $N$ is a minimal normal subgroup of $W$.
Let $M$~be any normal subgroup of~$W$ and suppose $M \nleq N$.  
Then the image of~$M$ in~$G$ is nontrivial, so $[M,N] \neq 1$ as $G$~acts faithfully on~$\Omega$.  
Hence, $M \cap N \neq 1$, so $N \leq M$ by the minimality of~$N$.
This establishes that $N$~is the unique minimal normal subgroup of~$W$ and the remaining assertions follow immediately.
\end{proof}

The following example gives a construction of a uniserial group that is a wreath product that does not arise from the previous lemma.

\begin{example}
  \label{ex:wreath_quasisimple}
  Let $T$~be a quasisimple finite group whose centre $\Centre{T}$~is cyclic
  of prime order~$p$.  Let $G$~be a uniserial group, say with chief
  series $G = G_{0} > G_{1} > \dots > G_{\ell} = 1$, and suppose that
  $G$~acts faithfully and transitively on the finite set $\Omega =
  \{1,\dots,n\}$.  Then the wreath product $T \wr_{\Omega} G$
  contains an elementary abelian $p$\nbd subgroup that is central in
  its base group and that, when viewed as a $\F_{p}G$\nbd module, is
  isomorphic to the permutation module~$\F_{p}^{n}$.  Let $V$~be a
  uniserial quotient of the module~$\F_{p}^{n}$.  There is therefore a
  quotient~$H$ of $T \wr_{\Omega} G$ that is the semidirect product of
  a normal subgroup~$M$ by~$G$ where $M$~is perfect, $M/\Centre{M}
  \cong (T/\Centre{T})^{n}$ and $\Centre{M}$~is isomorphic to~$V$ as a
  $\F_{p}G$\nbd module.  Let us identify~$\Centre{M}$ with~$V$.  This
  group~$H$ then has a normal series
  \[
  H = M \semidirect G_{0} > M \semidirect G_{1} > \dots > M
  \semidirect G_{\ell} = M > \Centre{M} = V = V_{0} > V_{1} > \dots >
  V_{k} = 1
  \]
  where $V_{0} > V_{1} > \dots > V_{k}$ is the unique composition series of the module~$V$.
  
  We claim that $H$ is uniserial.
  To prove this, let $N$ be a normal subgroup of~$H$.  
  If $N \leq \Centre{M}$,
  then it is an $\F_{p}G$\nbd submodule of~$V$, so $N = V_{j}$ for
  some~$j$.  Suppose instead that $N \nleq \Centre{M}$.  Since
  $T/\Centre{T}$~is a nonabelian simple group and $G$~acts faithfully
  on~$\Omega$, it follows that $N$~acts nontrivially on the quotient
  $M/\Centre{M} \cong (T/\Centre{T})^{n}$.  Therefore $[M,N] \nleq
  \Centre{M}$ and we deduce that $(M \cap N) \Centre{M} = M$.  As
  $M$~is perfect, $M = M' = \bigl( (M \cap N) \Centre{M} \bigr)' = (M \cap
  N)' \leq N$.  Now the fact that $G$~is uniserial ensures that $N = M
  \semidirect G_{i}$ for some~$i$.  In conclusion, this shows that the
  constructed quotient~$H$ is a uniserial group.

  As a specific example, let $q$ be a prime power, let $d \geq 2$, let $p$ be a prime divisor of $\gcd(d, q-1)$ and let $T$ be the quotient of $\SL_d(q)$ by the central subgroup of order $\gcd(d, q-1)/p$.
  Then $T$~is a quasisimple group with $\Centre{T} \cong C_{p}$. 
  If $n \geq 5$ is divisible by~$p$, then the permutation module $V = \F_{p}^{n}$ for the symmetric group~$S_{n}$ is uniserial and we deduce that $T \wr S_{n}$ is uniserial with chief series
  \[
  T \wr S_{n} > T \wr A_{n} > T^{n} > \Centre{T}^{n} = V > V_{1} > V_{2} > 1
  \]
  where $V_{1}$~and~$V_{2}$ are the submodules given in Example~\ref{ex:affine_symmetric}.
\end{example}

\begin{paragraph}{Acknowledgements}
The first author is an EPSRC Postdoctoral Fellow (EP/X011879/2). 
The authors thank Chris Parker for drawing their attention to this question, Andrea Lucchini for helpful suggestions on how the work in a previous version of this paper could be set in a more general context and the anonymous referee for their careful reading of the paper.
In order to meet institutional and research funder open access requirements, any accepted manuscript arising shall be open access under a Creative Commons Attribution (CC BY) reuse licence with zero embargo.
No new data were created. 
\end{paragraph}

\begin{multicols}2
\noindent Scott Harper \newline
School of Mathematics \newline
University of Birmingham \newline
Birmingham, B15 2TT, UK \newline
\texttt{s.harper.3@bham.ac.uk}

\noindent Martyn Quick \newline
School of Mathematics and Statistics \newline
University of St Andrews \newline
St Andrews, KY16 9SS, UK \newline
\texttt{mq3@st-andrews.ac.uk}
\end{multicols}

\end{document}